\documentclass[review]{elsarticle}

\usepackage[latin1]{inputenc} 
\usepackage{listings, tcolorbox}

\usepackage{lineno,hyperref}
\usepackage{amssymb}
\usepackage{color}
\usepackage{amsmath}
\usepackage{tikz}
\usepackage{amsfonts}
\usepackage{mathrsfs}
 \usepackage{amsthm}

\usepackage{bbding}
\usepackage{mathrsfs}
\usepackage{amsmath, amsfonts, amssymb, mathrsfs, txfonts}
\usepackage{graphicx, subfigure}
\usepackage{wasysym}
\usepackage{color}
\usepackage{bm}
\usepackage{enumerate}

\usepackage{pifont}
\usepackage{txfonts}
\usepackage{amsmath}
\usepackage{graphicx}
\usepackage{amsfonts}
\usepackage{amssymb}
\usepackage{mathrsfs,psfrag,eepic,epsfig}
\usepackage{epstopdf}

\modulolinenumbers[5]

\journal{Journal of \LaTeX\ Templates}

\makeatletter \@addtoreset{equation}{section}

\newtheorem{thm}{Theorem}[section]
\newtheorem{cor}[thm]{Corollary}
\newtheorem{lem}[thm]{Lemma}
\newtheorem{prop}[thm]{Proposition}
\theoremstyle{definition}

\newtheorem{rem}[thm]{Remark}
\newtheorem{assum}[thm]{Assumption}
\newtheorem{RHP}[thm]{Riemann-Hilbert Problem}

\renewcommand{\baselinestretch}{1.25}
\begin{document}

\begin{frontmatter}

\title{Soliton resolution for the complex short pulse equation with weighted Sobolev initial data \tnoteref{mytitlenote}}
\tnotetext[mytitlenote]{
Corresponding author.\\
\hspace*{3ex}\emph{E-mail addresses}: sftian@cumt.edu.cn,
shoufu2006@126.com (S. F. Tian) }

\author{Zhi-Qiang Li, Shou-Fu Tian$^{*}$ and Jin-Jie Yang}
\address{
School of Mathematics, China University of Mining and Technology,  Xuzhou 221116, People's Republic of China
}

\begin{abstract}
We employ the $\bar{\partial}$-steepest descent method in order to investigate
the Cauchy problem of the complex short pulse (CSP) equation with initial conditions in weighted Sobolev space $H^{1,1}(\mathbb{R})=\{f\in L^{2}(\mathbb{R}): f',xf\in L^{2}(\mathbb{R})\}$. The long time asymptotic behavior of the solution $u(x,t)$ is derived in a fixed space-time cone $S(x_{1},x_{2},v_{1},v_{2})=\{(x,t)\in\mathbb{R}^{2}: y=y_{0}+vt, ~y_{0}\in[y_{1},y_{2}], ~v\in[v_{1},v_{2}]\}$. Based on the resulting asymptotic behavior, we prove the solution resolution  conjecture of the CSP equation which includes the soliton term confirmed by $N(I)$-soliton on discrete spectrum and the $t^{-\frac{1}{2}}$ order term on continuous spectrum with residual error up to $O(t^{-1})$.
\end{abstract}


\end{frontmatter}

\tableofcontents
\newpage

\section{Introduction}
In nonlinear optics,  the well-known nonlinear Schr\"{o}dinger (NLS) equation can be used to model the pulse propagation in optical fibers\cite{NLS-optic}.  It is effective that the NLS equation is used to approximate the Maxwell's equations \cite{NLS-Maxwell} as the amplitude changes slowly. Therefore, more attention is paid to the  research of NLS-type equations \cite{Tian-PAMS}-\cite{Wangds-2019-JDE}. However, when the pulse becomes shorter, i.e., the width of optical pulse in the order of femtosecond($10^{-15}s$), it is not suitable to use the NLS equation continuously for describing the optical pulse propagation \cite{NLS-femtosecond}. In 2004, Sch\"{a}fer and Wayne proposed the short pulse (SP) equation \cite{SP-Eq}
\begin{align}\label{SP-equation}
q_{xt}(x,t)=q(x,t)+\frac{1}{6}(q^{3}(x,t))_{xx},
\end{align}
which can be used to describe the ultra-short optical pulse and approximate the corresponding solution of the Maxwell's equations more effectively. More importantly, the SP equation \eqref{SP-equation} can be viewed as the short-wave limit of the modified Camassa-Holm (CH) equation \cite{mCH-1}-\cite{mCH-2}
\begin{align}\label{mCH-equation}
m_t+\left((u^{2}-u^{2}_{x})m\right)_{x}+2u_{x}=0.
\end{align}
That means the SP equation can be transformed into mCH equation via applying a transformation.
Since the CH equation and modified CH equation have rich mathematical structure and  properties \cite{Constantin-1}-\cite{Fokas-mCH}, it is meaningful to study the SP equation \eqref{SP-equation}.
Regrettably, it is noted that $q(x,t)$ is a real-valued function in Eq.\eqref{SP-equation} which implies that the one-soliton solution of the SP equation \eqref{SP-equation} possesses no physical interpretation although the SP equation \eqref{SP-equation} is derived from the physical background \cite{SP-no-phy-1,SP-no-phy-2}. In order to study the solution of SP equation in the actual physical context,  Feng proposed the so-called complex short pulse equation (CSP) equation\cite{CSP-equation-Feng}
\begin{align}\label{1.1}
u_{xt}+u+\frac{1}{2}(|u|^{2}u_{x})_{x}=0,
\end{align}
where $u(x,t)$ is a complex-valued function in 2015. It is worth noting that amplitude and phase can be described by using the complex-valued function. Thus, it is more effective to use CSP equation to describe the ultra-short optical pulse propagation in optical fibers. Moreover, like SP equation \cite{Xu-SP-JDE}, the CSP equation \eqref{1.1} also admits a Wadati-Konno-Ichikawa (WKI)-type Lax pair \cite{CSP-equation-Feng,CSP-equation-Lax}. Then, lots of work for the CSP equation \eqref{1.1} have been done. For example, via applying Hirota method and Darboux transformation method, the soliton solution, multi-breather and higher-order rogue wave solution of the CSP equation \eqref{1.1} are reported \cite{CSP-equation-Feng,CSP-equation-rogue}.
Moreover, the conservation laws of the CSP equation \eqref{1.1} have been studied in \cite{CSP-conserv-Feng}. From Lax pair representation \eqref{2.1}, the following formula can be obtained by employing the transformation $\Gamma= \psi_{2} \psi_{1}^{-1}$, i.e.,
\begin{align}\label{conserv-1}
2zu_{x}\Gamma=zu_{x}u_{x}^{*}-u_{x}(u^{-1}_{x}\cdot u_{x}\Gamma)_{x}-z(u_{x}\Gamma)^{2}.
\end{align}
Expanding $u_{x}\Gamma$ as follows
\begin{align*}
u_{x}\Gamma=\sum_{n=1}^{\infty}F_{n}z^{-n},
\end{align*}
and substituting it into Eq.\eqref{conserv-1}, it is easy to derive that $F_{n}$ satisfies the following recurrence relation
\begin{align*}
2F_{n}=u_{x}u_{x}^{*}\delta_{n,0}-u_{x}(u_{x}^{-1}F_{n-1})_{x}-\sum_{\ell=0}^{n}F_{\ell}F_{n-\ell}.
\end{align*}
The conserved density turns out to be
\begin{align*}
F_{0}&=-1+\sqrt{1+|u_{x}|^{2}},~~ F_{1}=-\frac{u_{x}u_{xx}^{-1}(-1+\sqrt{1+|u_{x}|^{2}})}{2\sqrt{1+|u_{x}|^{2}}} -\frac{u^{*}_{x}u_{xx}+u_{x}u^{*}_{xx}}{4(1+|u_{x}|^{2})},\\
F_{2}&=-\frac{u_{x}u_{xx}^{-1}F_{1}-F_{1,x}-F^{2}_{1}}{2\sqrt{1+|u_{x}|^{2}}}, \ldots.
\end{align*}
Then the  conserved quantities can be expressed as
\begin{align*}
I_{0}&=\int_{-\infty}^{+\infty}\left(-1+\sqrt{1+|u_{x}(x,t)|^{2}}\right)dx,\\
I_{1}&=\int_{-\infty}^{+\infty}\left(-\frac{u_{x}u_{xx}^{-1}(-1+\sqrt{1+|u_{x}|^{2}})}{2\sqrt{1+|u_{x}|^{2}}} -\frac{u^{*}_{x}u_{xx}+u_{x}u^{*}_{xx}}{4(1+|u_{x}|^{2})}\right)dx,\\
I_{2}&=\int_{-\infty}^{+\infty}\left(-\frac{u_{x}u_{xx}^{-1}F_{1}-F_{1,x}-F^{2}_{1}}{2\sqrt{1+|u_{x}|^{2}}}
\right)dx,\ldots.
\end{align*}
In addition, applying the nonlinear steepest descent method of Defit and Zhou,  Xu and Fan \cite{Xu-CSP-JDE} shows the long time asymptotic behavior of the CSP equation \eqref{1.1} with residual error up to $O(\frac{\log t}{t})$.

In this work, we employ $\bar{\partial}$-steepest descent method to investigate the soliton resolution for the CSP equation with the initial value condition
\begin{align}\label{1.3}
  u(x,0)=u_{0}(x)\in H^{1,1}(\mathbb{R}),
\end{align}
where
\begin{align}\label{Sobolev-space}
 H^{1,1}(\mathbb{R})=\{f\in L^{2}(\mathbb{R}): f',xf\in L^{2}(\mathbb{R})\}.
\end{align}
It is interesting that compared with the result reported in \cite{Xu-CSP-JDE}, our work has a more obvious advantage in the research of long time asymptotic behavior for the CSP equation \eqref{1.1}. The accuracy of our asymptotic result   can reach $O(t^{-1})$, which
 cannot be achieved in the previous work \cite{Xu-CSP-JDE}.

Since Manakov first paid attention to the long time asymptotic behavior of nonlinear evolution equations \cite{Manakov-1974}, the research of it has been widely concerned. In 1976, Zakharov and Manakov derived the long time asymptotic solutions of NLS equation with decaying initial value \cite{Zakharov-1976}. In 1993, Defit and Zhou developed a nonlinear steepest descent method which can be used to systematically study the long time asymptotic behavior of nonlinear evolution equations \cite{Deift-1993}. After years of unremitting research by scholars, the nonlinear steepest descent method has been improved. An example is that when the initial value is smooth and decay fast enough, the error term is $O(\frac{\log t}{t})$ which is shown in \cite{Deift-1994-1, Deift-1994-2}. And the work \cite{Deift-2003} shows that the error term is $O(t^{-(\frac{1}{2}+\iota)})$ for any $0<\iota<\frac{1}{4}$ when the initial value belongs to the weighted Sobolev space \eqref{Sobolev-space}.

In recent years, combining steepest descent with $\bar{\partial}$-problem, McLaughlin and Miller \cite{McLaughlin-1, McLaughlin-2}, developed a $\bar{\partial}$-steepest descent method to study the asymptotic of orthogonal polynomials. Then, this method was successfully used to investigate defocusing NLS equation with finite mass initial data \cite{Dieng-2008} and with finite density initial data \cite{Cuccagna-2016}. It should be pointed out that different from the nonlinear steepest descent method, the delicate estimates involving $L^{p}$ estimates of Cauchy projection operators can be avoided by using $\bar{\partial}$-steepest descent method. Also, the work in \cite{Dieng-2008} shows that the error term is $O(t^{-\frac{3}{4}})$ when the initial value belongs to the weighted Sobolev space \eqref{Sobolev-space}. Therefore, a series of great work has been done by applying $\bar{\partial}$-steepest descent method \cite{AIHP}-\cite{Faneg-3}.

In \cite{Yang-SP}, Yang and Fan give the long time asymptotic behavior of the solution $q(x,t)$ of the SP equation \eqref{SP-equation} via applying the $\bar{\partial}$-steepest descent method. In this work, we extend above results to derive the long time asymptotic behavior of the solution $u(x,t)$ of the CSP equation \eqref{1.1}. It is worth noting that there are some differences from that on SP equation \eqref{SP-equation} which is shown in the following four aspects.

\begin{enumerate}[(I)]
  \item When we construct the Riemann-Hilbert problem (RHP) corresponding to the initial value problem for the CSP equation \eqref{1.1}, an improved transformation need to be introduced to guarantee that the eigenfunctions tend to the identity matrix as the spectral parameter $z\rightarrow\infty$. An obvious result is that there exists an exponential term in the solution $u(x,t)$ which is shown in \eqref{u-sol}.
  \item Compared with the case in SP equation, the symmetry condition, i.e.,
  \begin{align}\label{sym-z-z}
    M(x,t,-z)=\sigma_{2}M(x,t,z)\sigma_{2},
  \end{align}
  does not exist when we construct the RHP corresponding to the CSP equation \eqref{1.1}.
  \item Since the symmetry condition \eqref{sym-z-z} does not exist, it is necessary to analyze the local model problem around the phase points $z=\pm z_{0}$ respectively, see subsection \ref{section-local-model}. Also due to this reason, the final results we have obtained in this work is essentially different from the case for the SP equation.
  \item Due to the difference between the Lax pair of the CSP equation and SP equation, $\theta(z)$, which is defined in section \ref{section-Conjugation}, is different from the the case for the SP equation which will have influence on the analysis of the $\bar{\partial}$-RH problem for $M^{(3)}(z)$ which is defined in \eqref{delate-pure-RHP}. We need to take some different scaling techniques to investigate the estimates of $M^{(3)}$, see section \ref{section-Pure-dbar-RH}.
\end{enumerate}

\noindent \textbf{Our main result and remark of the soliton resolution conjecture for the CSP equation \eqref{1.1} are given as follows.}

\begin{thm}\label{Thm-1}
Suppose that the initial values $u_{0}(x)$ satisfy the Assumption \eqref{assum} and $u_{0}(x)\in H^{1,1}(\mathbb{R})$. Let $u(x,t)$ be the solution of CSP equation \eqref{1.1}. The scattering data is denoted as $\{r,\{z_{k},c_{k}\}_{k=1}^{N}\}$ which generated from the initial values $u_{0}(x)$. For fixed $y_{1},y_{2},v_{1},v_{2}\in \mathbb{R}$ with $x_{1}<x_{2}$, $v_{1}<v_{2}\in\mathbb{R}^{-}$, and $I=\{z:-\frac{1}{4v_{1}}<|z|^{2}<-\frac{1}{4v_{2}}\}$, $z^{2}_{0}=\frac{t}{4y}$, then as $t\rightarrow \infty$ and $(y,t)\in S(y_{1},y_{2},v_{1},v_{2})$ which is defined in \eqref{space-time-S}, the solution $u(x,t)$ can be expressed as
\begin{align}\label{9.2}
\begin{split}
u(x,t)e^{-2d}&=u(y(x,t),t)e^{-2d}\\
&=u_{sol}(y(x,t),t;\sigma_{d}(I))T^{2}(0)(1+T_{1})-it^{-\frac{1}{2}}f^{\pm}_{12}+O(t^{-1}),\\
y(x,t)=x-&c_{+}(x,t,\sigma_{d}(I))-iT_{1}^{-1}-it^{-\frac{1}{2}}f^{\pm}_{11}+O(t^{-1}).
\end{split}
\end{align}
Here, $u_{sol}(x,t;\hat{\sigma}_{d}(I))$ is the $N(I)$ soliton solution, $T(z)$ is defined in \eqref{4.5}, and
\begin{align*}
f^{\pm}_{12}=\frac{1}{i\sqrt{z_{0}}}
&[M^{(out)}(0)^{-1}(M^{(out)}(z_0)^{-1}M_{1}^{pc,\pm}(z_{0})M^{(out)}(z_0)\\&+
M^{(out)}(-z_0)^{-1}M_{1}^{(pc),\pm}(-z_{0})M^{(out)}(-z_0))M^{(out)}(0)]_{12},\\
f^{\pm}_{11}=\frac{1}{i\sqrt{z_{0}}}
&[M^{(out)}(0)^{-1}(M^{(out)}(z_0)^{-1}M_{1}^{pc,\pm}(z_{0})M^{(out)}(z_0)\\&+
M^{(out)}(-z_0)^{-1}M_{1}^{(pc),\pm}(-z_{0})M^{(out)}(-z_0))M^{(out)}(0)]_{11}.
\end{align*}
\end{thm}

\begin{rem}
Theorem \ref{Thm-1} need the condition $u_{0}(x)\in H^{1,1}(\mathbb{R})$ so that the inverse scattering transform possesses  well mapping properties. Also the condition $u_{0}(x)\in H^{1,1}(\mathbb{R})$ guarantees that there exists no discrete spectrum on the real axis. It is noted that the asymptotic results only depend on the $H^{1}(\mathbb{R})$ norm of $r$, therefore, for any $u_{0}(x)\in H^{1,1}(\mathbb{R})$ admitting the Assumption \eqref{assum}, the process of the large-time analysis and calculations shown in this work is unchanged.
\end{rem}

\noindent \textbf{Organization of the rest of the work}

In section 2, based on the Lax pair of the CSP equation, we introduce two kinds of eigenfunctions to deal with the spectral singularity. Also, the analytical, symmetries and asymptotic properties are analyzed.
In section 3, using similar ideas to \cite{Xu-CSP-JDE}, the RHP for $M(z)$ is constructed for the CSP equation with initial problem.
In section 4, in order to obtain a new RHP for $M^{(1)}(z)$ that its jump matrix can be decomposed into two triangle matrices near the phrase point $z=\pm z_{0}$, we introduce the matrix function $T(z)$ to define the new RHP.
In section 5, we make the continuous extension of the jump matrix off the real axis by introducing a matrix function $R^{(2)}(z)$ and get a mixed $\bar{\partial}$-Riemann-Hilbert(RH) problem.
In section 6, we decompose the mixed $\bar{\partial}$-RH problem into two parts which are a model RH problem with $\bar{\partial}R^{(2)}=0$ and a pure $\bar{\partial}$-RH problem with $\bar{\partial}R^{(2)}\neq0$, respectively, i.e., $M^{(2)}_{RHP}$ and $M^{(3)}$.
In section 7, we solve the model RH problem $M^{(2)}_{RHP}$  via an outer model $M^{(out)}(z)$ for the soliton part and inner model $M^{(\pm z_{0})}$ near the phase point $\pm z_{0}$ which can be solved by matching  parabolic cylinder model problem respectively. Also, the error function $E(z)$ with a small-norm RH problem is obtained.
In section 8, the pure $\bar{\partial}$-RH problem for $M^{(3)}$ is studied.
Finally, in section 9, we obtain the soliton resolution and long time asymptotic behavior of the CSP equation.

\section{The spectral analysis of CSP equation}
In order to study the soliton resolution of the initial value problem (IVP) for the CSP equation via applying $\bar{\partial}$-steepest descent method, we first construct a RHP based on the Lax pair of the CSP equation. The WKI-type Lax pair of the CSP equation reads
\begin{align}\label{2.1}
\psi_{x}(x,t,z)=U(x,t,z)\psi(x,t,z), ~~\psi_{t}(x,t,z)=V(x,t,z)\psi(x,t,z),
\end{align}
where
\begin{align*}
U(x,t,z)=izU_{1}=iz(\sigma_{3}+U_{0x}),\\
V(x,t,z)=-\frac{iz}{2}|u|^{2}U_{1}-\frac{1}{4iz}\sigma_{3}+\frac{1}{2}V_{0},
\end{align*}
with
\begin{align*}
U_{0}=\left(
        \begin{array}{cc}
          0 & u \\
          u^{*} & 0 \\
        \end{array}
      \right), ~~\sigma_{3}=\left(
                            \begin{array}{cc}
                              1 & 0 \\
                              0 & -1 \\
                            \end{array}
                          \right), ~~V_{0}=\left(
        \begin{array}{cc}
          0 & u \\
          -u^{*} & 0 \\
        \end{array}
      \right).
\end{align*}
The $u^{*}$ infers to the conjugate of the complex potential function $u$.

Generally, when we deal with the IVP of integrable equations, we just employ the $x$-part of the Lax pair base on the inverse scattering transform method. The $t$-part of Lax pair is used to control the time evolution of the scattering data. However, the Lax pair \eqref{2.1} of the CSP equation possesses two singularities, i.e., $z=0$ and $z=\infty$. Consequently, in order to recover the potential function $u(x,t)$, the $t$-part of Lax pair and the expansion of the eigenfunction as spectral parameter $z\rightarrow 0$. Therefore, we deal with the two singularities at $z=0$ and $z=\infty$ applying two different transformations in the following analysis.

\subsection{The case of singularity at z=0}
We first introduce a transformation
\begin{align}\label{2.2}
\psi(x,t;z)=\mu^{0}(x,t;z)e^{i(zx+\frac{1}{4iz}t)\sigma_{3}},
\end{align}
then, an equivalent Lax pair can be derived as
\begin{align}
\begin{split}
\mu^{0}_{x}&-iz[\sigma_{3},\mu^{0}]=U_{2}\mu^{0},\\
\mu^{0}_{t}&-\frac{i}{4z}[\sigma_{3},\mu^{0}]=V_{2}\mu^{0},\label{2.3}
\end{split}
\end{align}
where
\begin{align*}
U_{2}=izU_{0x},~~V_{2}=-\frac{iz}{2}|u|^{2}U_{1}+\frac{1}{2}V_{0},
\end{align*}
and $\mu^{0}=\mu^{0}(x,t;z)$. Additionally, $[A,B]$ means $AB-BA$ where $A$ and $B$ are $2\times2$ matrices. The Lax pair \eqref{2.3} can be written in full derivative form
\begin{align}\label{2.4}
d(e^{-i(zx+\frac{1}{4z}t)\hat{\sigma}_{3}}\mu^{0})=
e^{-i(zx+\frac{1}{4z}t)\hat{\sigma}_{3}}(U_{2}dx+V_{2}dt)\mu^{0},
\end{align}
where $e^{\hat{\sigma}_{3}}A=e^{\sigma_{3}}Ae^{-\sigma_{3}}$.
By selecting two special integration paths i.e., $(-\infty,t)\rightarrow(x,t)$ and $(+\infty,t)\rightarrow(x,t)$, on Eq.\eqref{2.4}, we define two eigenfunction $\mu^{0}_{\pm}(x,t;z)$ which can be derived as the following Volterra type integrals
\begin{align}
\begin{matrix}
\mu^{0}_{-}(x,t;z)=\mathbb{I}+
\int_{x}^{-\infty}e^{iz(x-y)\hat{\sigma}_{3}}U_{2}(y,t;z)\mu^{0}_{-}(y,t;z)dy,\\
\mu^{0}_{+}(x,t;z)=\mathbb{I}-
\int_{x}^{+\infty}e^{iz(x-y)\hat{\sigma}_{3}}U_{2}(y,t;z)\mu^{0}_{+}(y,t;z)dy.
\end{matrix}
\end{align}
Then we can derive the analytic property and asymptotic property of $\mu^{0}_{\pm}(x,t;z)$.

\begin{prop}
The properties of $\mu^{0}_{\pm}(x,t;z)$:
\begin{itemize}
  \item (Analytic property) It is assumed that $u(x)-u_{0}\in H^{1,1}(\mathbb{R})$. Then, $\mu^{0}_{-,1}, \mu^{0}_{+,2}$ are analytic in $C_{-}$ and $\mu^{0}_{-,2}, \mu^{0}_{+,1}$ are analytic in $C_{+}$. The $\mu^{0}_{\pm,j} (j=1,2)$ mean the $j$-th column of $\mu^{0}_{\pm}$.
  \item (Asymptotic property) The function $\mu^{0}_{\pm}(x,t;z)$ admit the following asymptotic expansions as $z\rightarrow0$,
      \begin{align}\label{2.5}
       \mu^{0}_{\pm}(x,t;z)=\mathbb{I}+\left(
                                         \begin{array}{cc}
                                           0 & iu(x,t) \\
                                           iu^{*}(x,t) & 0 \\
                                         \end{array}
                                       \right)z+O(z^{2}).
      \end{align}
\end{itemize}
\end{prop}
\subsection{The case of singularity at z=$\infty$}
Considering the singularity at $z=\infty$, we need to control the asymptotic behavior of eigenfunctions as $z\rightarrow\infty$. Thus, following the idea in \cite{Xu-CSP-JDE}, we introduce the transformation
\begin{align}\label{2.6}
\psi(x,t;z)=G(x,t)\phi e^{izp(x,t;z)\sigma_{3}},
\end{align}
where
\begin{align*}
  G(x,t)=\sqrt{\frac{\sqrt{m(x,t)}+1}{2\sqrt{m(x,t)}}}\left(
                                              \begin{array}{cc}
                                                1 & -\frac{\sqrt{m(x,t)}-1}{u^{*}_{x}(x,t)} \\
                                                \frac{\sqrt{m(x,t)}-1}{u_{x}(x,t)} & 1 \\
                                              \end{array}
                                            \right),\\
  m(x,t)=1+|u_{x}|^{2},~~
  p(x,t;z)=x-\int_{x}^{\infty}(\sqrt{m(s,t)}-1)ds+\frac{t}{4z^{2}}.
\end{align*}
Consequently, the CSP equation \eqref{1.1} can be written in conservation law form
\begin{align*}
\left(\sqrt{m(x,t)}\right)_{t}=-\frac{1}{2}\left(|u(x,t)|^{2}\sqrt{m(x,t)}\right)_{x},
\end{align*}
and the derivative of function $p(x,t;z)$ with respect to $x$ and $t$ can be derived as
\begin{align}\label{2.7}
p_{x}(x,t;z)=\sqrt{m(x,t)},~~p_{t}(x,t;z)=-\frac{1}{2}|u(x,t)|^{2}\sqrt{m(x,t)}+\frac{1}{4z^{2}}.
\end{align}
Also the equivalent Lax pair of $\psi(x,t;z)$ \eqref{2.1} is transformed into
\begin{align}\label{2.8}
\begin{split}
\phi_{x}-izp_{x}[\sigma_{3},\phi]=U_{3}\phi,\\
\phi_{t}-izp_{t}[\sigma_{3},\phi]=V_{3}\phi,
\end{split}
\end{align}
where
\begin{align*}
U_{3}=&-\left(
         \begin{array}{cc}
           \frac{u_{x}u^{*}_{xx}-u_{xx}u^{*}_{x}}{4\sqrt{m}(\sqrt{m}+1)} & \frac{(\sqrt{m}-1)u_{x}u^{*}_{xx}-(\sqrt{m}+1)u_{xx}u^{*}_{x}}{4mu_{x}^{*}} \\
           \frac{(\sqrt{m}+1)u_{x}u^{*}_{xx}-(\sqrt{m}-1)u_{xx}u^{*}_{x}}{4mu_{x}^{*}} & -\frac{u_{x}u^{*}_{xx}-u_{xx}u^{*}_{x}}{4\sqrt{m}(\sqrt{m}+1)} \\
         \end{array}
       \right),\\
V_{3}=&-\frac{1}{4iz}\frac{1}{\sqrt{m}}\sigma_{3}+\frac{1}{4iz}\frac{1}{\sqrt{m}}\left(
                                                                       \begin{array}{cc}
                                                                        0 & u_{x} \\
                                                                        u_{x}^{*} & 0 \\
                                                                       \end{array}
                                                                                 \right)
+\frac{1}{4iz}\sigma_{3}\\
&-\frac{1}{4\sqrt{m}}\left(
                       \begin{array}{cc}
                         u^{*}u_{x}-uu_{x}^{*} & -\frac{(\sqrt{m}+1)uu^{*}_{x}+(\sqrt{m}-1)u_{x}u^{*}}{u_{x}^{*}} \\
                         \frac{(\sqrt{m}-1)uu^{*}_{x}+(\sqrt{m}+1)u_{x}u^{*}}{u_{x}} & -u^{*}u_{x}+uu_{x}^{*} \\
                       \end{array}
                     \right)\\
&-\left(
         \begin{array}{cc}
           \frac{u_{x}u^{*}_{xt}-u_{xt}u^{*}_{x}}{4\sqrt{m}(\sqrt{m}+1)} & \frac{(\sqrt{m}-1)u_{x}u^{*}_{xt}-(\sqrt{m}+1)u_{xt}u^{*}_{x}}{4mu_{x}^{*}} \\
           \frac{(\sqrt{m}+1)u_{x}u^{*}_{xt}-(\sqrt{m}-1)u_{xt}u^{*}_{x}}{4mu_{x}^{*}} & -\frac{u_{x}u^{*}_{xt}-u_{xt}u^{*}_{x}}{4\sqrt{m}(\sqrt{m}+1)} \\
         \end{array}
       \right).
\end{align*}
Based on the Lax pair \eqref{2.8}, it is not hard to verify that the solutions of spectral problem do not approximate the identity matrix as $z\rightarrow\infty$ which  will cause difficulties in constructing RHP. Therefore, we need to introduce an improved transformation
\begin{align}\label{2.9}
\psi(x,t;z)=G(x,t)e^{d_{-}\hat{\sigma}_{3}}\mu(x,t;z)e^{-d_{+}\sigma_{3}} e^{izp(x,t;z)\sigma_{3}},
\end{align}
where
\begin{gather*}
d_{-}=\int_{-\infty}^{x}\frac{u_{xx}u^{*}_{x}-u_{x}u^{*}_{xx}}{4\sqrt{m}(\sqrt{m}+1)}(s,t)ds,
~~d_{+}=\int^{+\infty}_{x}\frac{u_{xx}u^{*}_{x}-u_{x}u^{*}_{xx}}{4\sqrt{m}(\sqrt{m}+1)}(s,t)ds,\\
d=d_{+}+d_{-}=\int_{-\infty}^{+\infty} \frac{u_{xx}u^{*}_{x}-u_{x}u^{*}_{xx}}{4\sqrt{m}(\sqrt{m}+1)}(s,t)ds.
\end{gather*}
Then, the equivalent Lax pair of $\psi(x,t;z)$ \eqref{2.1} can be written as
\begin{align}\label{2.29}
\begin{split}
\mu_{x}-izp_{x}[\sigma_{3},\mu]=e^{-d_{-}\hat{\sigma}_{3}}U_{4}\mu,\\
\mu_{t}-izp_{t}[\sigma_{3},\mu]=e^{-d_{-}\hat{\sigma}_{3}}V_{4}\mu,
\end{split}
\end{align}
where
\begin{align*}
U_{4}=&-\left(
         \begin{array}{cc}
           0 & \frac{(\sqrt{m}-1)u_{x}u^{*}_{xx}-(\sqrt{m}+1)u_{xx}u^{*}_{x}}{4mu_{x}^{*}} \\
           \frac{(\sqrt{m}+1)u_{x}u^{*}_{xx}-(\sqrt{m}-1)u_{xx}u^{*}_{x}}{4mu_{x}^{*}} & 0 \\
         \end{array}
       \right),\\
V_{4}=&-\frac{1}{4iz}(\frac{1}{\sqrt{m}}-1)\sigma_{3}+\frac{1}{4iz}\frac{1}{\sqrt{m}}\left(
                                                                       \begin{array}{cc}
                                                                        0 & u_{x} \\
                                                                        u_{x}^{*} & 0 \\
                                                                       \end{array}
                                                                                 \right)\\
&-\frac{1}{4\sqrt{m}}\left(
                       \begin{array}{cc}
                         0 & -\frac{(\sqrt{m}+1)uu^{*}_{x}+(\sqrt{m}-1)u_{x}u^{*}}{u_{x}^{*}} \\
                         \frac{(\sqrt{m}-1)uu^{*}_{x}+(\sqrt{m}+1)u_{x}u^{*}}{u_{x}} & 0 \\
                       \end{array}
                     \right)\\
&-\left(
         \begin{array}{cc}
           0 & \frac{(\sqrt{m}-1)u_{x}u^{*}_{xt}-(\sqrt{m}+1)u_{xt}u^{*}_{x}}{4mu_{x}^{*}} \\
           \frac{(\sqrt{m}+1)u_{x}u^{*}_{xt}-(\sqrt{m}-1)u_{xt}u^{*}_{x}}{4mu_{x}^{*}} & 0 \\
         \end{array}
       \right).
\end{align*}
Furthermore, Eq.\eqref{2.29} can be written in full derivative form
\begin{align}\label{fullDerivative}
d(e^{-izp(x,t;z)\hat{\sigma}_{3}}\mu)=e^{-izp(x,t;z)\hat{\sigma}_{3}} e^{-d_{-}\hat{\sigma}_{3}}(U_{4}dx+V_{4}dt)m,
\end{align}
from which we can derive two Volterra type integrals
\begin{align}\label{2.10}
\begin{matrix}
\mu_{-}(x,t;z)=\mathbb{I}+
\int_{x}^{-\infty}e^{iz[p(x,t;z)-p(s,t;z)]\hat{\sigma}_{3}}e^{-d_{-}\hat{\sigma}_{3}}
U_{4}(s,t;z)\mu_{-}(s,t;z)ds,\\
\mu_{+}(x,t;z)=\mathbb{I}-
\int_{x}^{+\infty}e^{iz[p(x,t;z)-p(s,t;z)]\hat{\sigma}_{3}}e^{-d_{-}\hat{\sigma}_{3}}
U_{4}(y,t;z)\mu_{+}(s,t;z)ds.
\end{matrix}
\end{align}
Based on the definition of $\mu(x,t;z)$ and the above integrals \eqref{2.10}, we can derive the properties of $\mu(x,t;z)$ including analytic, symmetry and asymptotic behavior properties.
\begin{prop}\label{2-1}
The properties of $\mu(x,t;z)$:
\begin{itemize}
  \item (Analytic property)
  It is assumed that $u(x)-u_{0}\in H^{1,1}(\mathbb{R})$. Then, $\mu_{-,1}, \mu_{+,2}$ are analytic in $C_{-}$ and $\mu_{-,2}, \mu_{+,1}$ are analytic in $C_{+}$. The $\mu_{\pm,j} (j=1,2)$ mean the $j$-th column of $\mu_{\pm}$.
  \item (Symmetry property)
  The symmetry of the eigenfunctions $\mu_{\pm}(x,t;z)$ can be shown as
  \begin{align}\label{2.11}
    \mu^{*}_{\pm}(x,t;z^{*})=\sigma_{2}\mu_{\pm}(x,t;z)\sigma_{2},
  \end{align}
  where $\sigma_{2}=\left(
                      \begin{array}{cc}
                        0 & -i \\
                        i & 0 \\
                      \end{array}
                    \right)$.
  \item (Asymptotic property for $z\rightarrow\infty$)
  The function $\mu_{\pm}(x,t;z)$ admit the following asymptotic expansions as $z\rightarrow\infty$,
      \begin{align}\label{2.12}
       \mu_{\pm}(x,t;z)=\mathbb{I}+O(z^{-1}).
      \end{align}
\end{itemize}
\end{prop}
\subsection{The scattering matrix}
Considering the fact that the eigenfunctions $\mu_{\pm}(x,t;z)$ are two fundamental matrix solutions of Eq.\eqref{2.29} for $z\in \mathbb{R}$, there exists a matrix $S(z)$ that leads to
\begin{align}\label{2.13}
\mu_{-}(x,t;z)=\mu_{+}(x,t;z)e^{izp(x,t;z)\hat{\sigma}_{3}}S(z),
\end{align}
where $S(z)=(s_{ij}(z))~(i,j=1,2)$ is independent of the variable $x$ and $t$.
Based on the Abel's theorem 
and the properties of $\mu_{\pm}(x,t;z)$ that shown in Proposition \ref{2-1}, the properties of $S(z)$ can be derived.
\begin{prop}\label{2-2}
The properties of $S(z)$:
\begin{itemize}
  \item (Analytic property)
  $s_{11}$ is analytic in $\mathbb{C}^{-}$, and $s_{22}$ is analytic in $\mathbb{C}^{+}$.
  \item (Symmetry property)
  The symmetry of the elements of the scattering matrix $S(z)$ can be shown as
  \begin{align}\label{2.21}
    s_{11}(z)=s^{*}_{22}(z^{*}),~~s_{12}(z)=-s^{*}_{21}(z^{*}).
  \end{align}
  \item (Asymptotic property for $z\rightarrow\infty$)
  The element $s_{11}(z)$ admit the following asymptotic expansions as $z\rightarrow0$,
      \begin{align}\label{2.22}
       s_{22}(z)=e^{d}\left(1+izc-\frac{c^{2}}{2}z^{2}+O(z^{3})\right).
      \end{align}
\end{itemize}
\end{prop}
\subsection{The connection between $\mu_{\pm}(x,t;z)$ and $\mu^{0}_{\pm}(x,t;z)$}
In the following analysis, we will use the eigenfunctions $\mu_{\pm}(x,t;z)$ to construct the matrix $M(x,t;z)$ and further formulate a RHP. It is worth noting that the asymptotic behavior of $\mu_{\pm}(x,t;z)$ as $z\rightarrow0$ plays an important role in constructing the solution $u(x,t)$. Thus, the connection between $\mu_{\pm}(x,t;z)$ and $\mu^{0}_{\pm}(x,t;z)$ is necessary.

From Eq.\eqref{2.2} and Eq.\eqref{2.9}, we can derive that
\begin{align}\label{2.23}
\mu_{\pm}(x,t;z)=e^{-d_{-}\sigma_{3}}G^{-1}\mu^{0}_{\pm}(x,t;z)
e^{i(zx+\frac{1}{4z}t)\sigma_{3}}C_{\pm}(z)e^{-izp(x,t;z)\sigma_{3}}e^{d\sigma_{3}},
\end{align}
where $C_{\pm}(z)$ are independent of $x$ and $t$.
Let $x\rightarrow\infty$, from Eq.\eqref{2.23},  $C_{\pm}(z)$ can be solved as
\begin{align*}
C_{+}(z)=\mathbb{I},~~ C_{-}(z)=e^{-d\sigma_{3}}e^{-izc\sigma_{3}},
\end{align*}
where $c=\int^{+\infty}_{-\infty}(\sqrt{m(x,t)}-1)dx$ is a quantity conserved under the dynamics governed by Eq.\eqref{1.1}. Then, the connection between $\mu_{\pm}(x,t;z)$ and $\mu^{0}_{\pm}(x,t;z)$ can be obtained as
\begin{align}\label{2.14}
\begin{split}
\mu_{-}(x,t;z)=e^{-d_{-}\sigma_{3}}G^{-1}(x,t)\mu^{0}_{-}(x,t;z)
e^{-iz\int^{x}_{-\infty}(\sqrt{m(s,t)}-1)ds\sigma_{3}},\\
\mu_{+}(x,t;z)=e^{-d_{-}\sigma_{3}}G^{-1}(x,t)\mu^{0}_{+}(x,t;z)
e^{iz\int^{+\infty}_{x}(\sqrt{m(s,t)}-1)ds\sigma_{3}}e^{d\sigma_{3}}.
\end{split}
\end{align}

\section{The formulation of a RHP}
\begin{assum}\label{assum}
In the following analysis, we make the assumption to avoid the many pathologies possible, i.e.,
\begin{itemize}
  \item For $z\in\mathbb{R}$, no spectral singularities exist, i.e, $s_{22}(z)\neq0$;
  \item Suppose that $s_{22}(z)$ possesses $N$ zero points, denoted as $\mathcal{Z}=\{(z_{j},Im~z_{j}>0)^{N}_{j=1}\}$.
  \item The discrete spectrum is simple, i.e., if $z_{0}$ is the zero of $s_{22}(z)$, then $s'_{22}(z_{0})\neq0$.
\end{itemize}
\end{assum}

Now, we introduce a sectionally  meromorphic matrices
\begin{align}\label{Matrix}
\tilde{M}(x,t;z)=\left\{\begin{aligned}
&\tilde{M}^{+}(x,t;z)=\left(\mu_{+,1}(x,t;z),\frac{\mu_{-,2}(x,t;z)}{s_{22}(z)}\right), \quad z\in \mathbb{C}^{+},\\
&\tilde{M}^{-}(x,t;z)=\left(\frac{\mu_{-,1}(x,t;z)}{s_{11}(z)},\mu_{+,2}(x,t;z)\right), \quad z\in \mathbb{C}^{-},
\end{aligned}\right.
\end{align}
where $\tilde{M}^{\pm}(x,t;z)=\lim\limits_{\varepsilon\rightarrow0^{+}}\tilde{M}(x,t;z\pm i\varepsilon),~\varepsilon\in\mathbb{R}$, and  reflection coefficients
\begin{gather}
r(z)=\frac{s_{12}(z)}{s_{22}(z)},~~
 \frac{s_{21}(z)}{s_{11}(z)}=-\frac{s^{*}_{12}(z^{*})}{s^{*}_{22}(z^{*})}
 =-r^{*}(z^{*})=-r^{*}(z),~~z\in\mathbb{R}.
\end{gather}

Based on the above analysis, the matrix function $\tilde{M}(x,t;z)$ admits the following matrix RHP.
\begin{RHP}\label{RH-1}
Find an analysis function $\tilde{M}(x,t;z)$ with the following properties:
\begin{itemize}
  \item $\tilde{M}(x,t;z)$ is meromorphic in $C\setminus\mathbb{R}$;
  \item $\tilde{M}^{*}(x,t;z^{*})=\sigma_{2}\tilde{M}(x,t;z)\sigma_{2}$;
  \item $\tilde{M}^{+}(x,t;z)=\tilde{M}^{-}(x,t;z)V(x,t;z)$,~~~$z\in\mathbb{R}$,
  where \begin{align}\label{J-Matrix-1}
V(x,t;z)=\left(\begin{array}{cc}
                   1 & r(z)e^{2izp} \\
                   r^{*}(z)e^{-2izp} & 1+|r(z)|^{2}
                 \end{array}\right);
\end{align}
  \item $\tilde{M}(x,t;z)=\mathbb{I}+O(z^{-1})$ as $z\rightarrow\infty$;
\end{itemize}
\end{RHP}
\begin{rem}
By referring to the Zhou's vanishing lemma, the existence of the solutions of RHP \ref{RH-1} for $(x,t)\in\mathbb{R}^{2}$ is guaranteed. According to a consequence of Liouville's theorem, we know that if a solution exists, it is unique.
\end{rem}

Next, in order to reconstruct the solution $u(x,t)$, the asymptotic behavior of $\tilde{M}(x,t;z)$ as $z\rightarrow 0$ need to be taken into account, i.e.,
\begin{align}\label{3.1}
\tilde{M}(x,t;z)=e^{-d_{-}\sigma_{3}}G^{-1}(x,t)\left[\mathbb{I}+z\left(ic_{+}\sigma_{3}
+i\left(
  \begin{array}{cc}
    0 & u \\
    u^{*} & 0 \\
  \end{array}
\right)\right)+O(z^{2})\right]e^{d\sigma_{3}},~~z\rightarrow0,
\end{align}
where $c_{+}(x,t)=\int^{+\infty}_{x}(\sqrt{m(s,t)}-1)ds$.
Since $p(x,t;z)$ that appears in jump matrix \eqref{J-Matrix} is not clear, it is still hard to obtain the solution $u(x,t)$. Thus, introducing a transformation
\begin{align}\label{3.2}
y(x,t)=x-\int^{+\infty}_{x}(\sqrt{m(s,t)}-1)ds=x-c_{+}(x,t),
\end{align}
the jump matrix can be expressed explicitly. However, we can just obtain the solution $u(x,t)$ only in implicit form: it will be given in terms of functions in the new scale, whereas the original scale will also be given in terms of functions in the new scale. We further define that \begin{align*}
       \tilde{M}(x,t;z)=M(y(x,t),t;z),
     \end{align*}
then, the $M(y(x,t),t;z)$ admits the following matrix RHP.
\begin{RHP}\label{RH-2}
Find an analysis function $M(y,t;z)$ with the following properties:
\begin{itemize}
  \item $M(y,t;z)$ is meromorphic in $C\setminus\mathbb{R}$;
  \item $M^{*}(y,t;z^{*})=\sigma_{2}M(x,t;z)\sigma_{2}$;
  \item $M^{+}(y,t;z)=M^{-}(y,t;z)V(x,t;z)$,~~~$z\in\mathbb{R}$,
  where \begin{align}\label{J-Matrix}
V(x,t;z)=e^{i\left(zy+\frac{t}{4z}\right)\hat{\sigma}_{3}}\left(\begin{array}{cc}
                   1 & r(z) \\
                   r^{*}(z) & 1+|r(z)|^{2}
                 \end{array}\right);
\end{align}
  \item $M(y,t;z)=\mathbb{I}+O(z^{-1})$ as $z\rightarrow\infty$;
\end{itemize}
\end{RHP}
Based on the Assumption \ref{assum}, Eq.\eqref{2.23} and Proposition \ref{2-1} , there exists norming constants $b_{j}$ such that
\begin{equation}
\mu_{-,2}(z_{j})=b_{j}e^{2i(z_{j}y+\frac{t}{4z_{j}})}\mu_{+,1}(z_{j});~
\mu_{-,1}(z^{*}_{j})=-b^{*}_{j}e^{-2i(z^{*}_{j}y+\frac{t}{4z^{*}_{j}})}\mu_{+,2}(z^{*}_{j}),\notag
\end{equation}
Then, the residue condition of $M(y,t;z)$ can be shown as
\begin{align}\label{2.32}
\mathop{Res}_{z=z_{j}}M=\lim_{z\rightarrow z_{j}}M\left(\begin{array}{cc}
                   0 & c_{j}e^{2i(z_{j}y+\frac{t}{4z_{j}})} \\
                   0 & 0
                 \end{array}\right),~~
\mathop{Res}_{z=z^{*}_{j}}M=\lim_{z\rightarrow z^{*}_{j}}M\left(\begin{array}{cc}
                   0 & 0 \\
                   -c^{*}_{j}e^{-2i(z^{*}_{j}y+\frac{t}{4z^{*}_{j}})} & 0
                 \end{array}\right).
\end{align}
where $c_{j}=\frac{b_{j}}{s'_{22}(z_{j})}$.

In terms of the solution of the RHP \ref{RH-2}, Proposition \ref{2-1} and Eq.\eqref{3.1}, the solution $u(x,t)$ can be derived as $u(x,t)=u(y(x,t),t)$, where
\begin{align}\label{u-sol}
\begin{split}
e^{-2d}u(y,t)=\lim_{z\rightarrow0}\frac{\left(M^{-1}(y,t;0)M(y,t;z)\right)_{12}}{iz},\\
x(y,t)=y+\lim_{z\rightarrow0}\frac{\left(M^{-1}(y,t;0)M(y,t;z)\right)_{11}-1}{iz}.
\end{split}
\end{align}

\section{Conjugation}\label{section-Conjugation}
In this section, our main purpose is to re-normalize the Riemann-Hilbert problem\eqref{RH-2}. Therefore, we will establish a transformation $M\mapsto M^{(1)}$ by introducing a function.

In jump matrix \eqref{J-Matrix}, the oscillation term is $e^{2i(zy+\frac{t}{4z})}$ which can be denoted as
\begin{align}\label{4.1}
e^{2i(zy+\frac{t}{4z})}=e^{2it\theta(z)},~~ \theta(z)=\frac{zy}{t}+\frac{1}{4z}.
\end{align}
Next, the phase points of $\theta(z)$ can be derived which can be denoted as $\pm z_{0}$ where
$z_{0}=\sqrt{\frac{t}{4y}}$. For the case that $\frac{t}{4y}<0$, the solution $u(x,t)$ of the initial problem \eqref{1.1} and \eqref{1.3} tends to $0$ fast decay as $t\rightarrow\infty$\cite{Xu-CSP-JDE}. Thus, we mainly pay attention to the case that $\frac{t}{4y}>0$.
Furthermore, $\theta(z)$ can be written as
\begin{align}\label{theta}
\theta(z)=\frac{z}{4}\left(\frac{1}{z_{0}^{2}}+\frac{1}{z^{2}}\right),
\end{align}
from which we can derive that
\begin{align}\label{4.2}
Re(2it\theta(z))=-2tIm~z\frac{|z|^{2}-z_{0}^{2}}{4z_{0}^{2}|z|^{2}}.
\end{align}

Then, we derive the decaying domains of the oscillation term.\\

\centerline{\begin{tikzpicture}[scale=0.7]
\path [fill=pink] (-1,0) -- (-9,0) to (-9,3) -- (-1,3);
\filldraw[white, line width=0.5](-3,0) arc (0:180:2);
\filldraw[pink, line width=0.5](-7,0) arc (-180:0:2);
\draw[->][thick](-9,0)--(-1,0);
\draw[fill] (-1,0)node[right]{$Rez$};
\draw[fill] (-5,0)node[below]{$0$} circle [radius=0.08];
\draw[fill] (-7.4,0)node[below]{$-z_{0}$};
\draw[fill] (-2.6,0)node[below]{$z_{0}$};
\draw[fill] (-7,0)circle [radius=0.08];
\draw[fill] (-3,0)circle [radius=0.08];
\draw[fill] (-7.5,1.5)node[above]{$|e^{2it\theta(z)}|\rightarrow\infty$};
\draw[fill] (-5,0)node[above]{$|e^{2it\theta(z)}|\rightarrow0$};
\draw[fill] (-5,-0.5)node[below]{$|e^{2it\theta(z)}|\rightarrow\infty$};
\draw[fill] (-7.5,-1.5)node[below]{$|e^{2it\theta(z)}|\rightarrow0$};
\draw[fill] (-5,3)node[above]{$t\rightarrow-\infty$};
\path [fill=pink] (1,0) -- (9,0) to (9,-3) -- (1,-3);
\filldraw[pink, line width=0.5](7,0) arc (0:180:2);
\filldraw[white, line width=0.5](3,0) arc (-180:0:2);
\draw[->][thick](1,0)--(9,0);
\draw[fill] (9,0)node[right]{$Rez$};
\draw[fill] (5,0)node[below]{0} circle [radius=0.08];
\draw[fill] (7.4,0)node[below]{$z_{0}$} ;
\draw[fill] (2.6,0)node[below]{$-z_{0}$} ;
\draw[fill] (7,0)circle [radius=0.08];
\draw[fill] (3,0)circle [radius=0.08];
\draw[fill] (7.5,1.5)node[above]{$|e^{2it\theta(z)}|\rightarrow0$};
\draw[fill] (5,0)node[above]{$|e^{2it\theta(z)}|\rightarrow\infty$};
\draw[fill] (5,-0.5)node[below]{$|e^{2it\theta(z)}|\rightarrow0$};
\draw[fill] (7.5,-1.5)node[below]{$|e^{2it\theta(z)}|\rightarrow\infty$};
\draw[fill] (5,3)node[above]{$t\rightarrow+\infty$};
\draw(-5,0) [black, line width=1] circle(2);
\draw(5,0) [black, line width=1] circle(2);
\end{tikzpicture}}
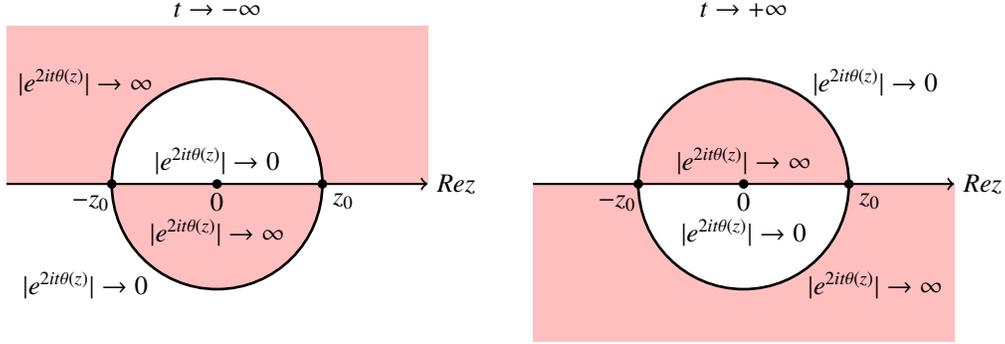
\centerline{\noindent {\small \textbf{Figure 1.} Exponential decaying domains.}}

To make the following analysis more convenient, we introduce some notations.
\begin{align}\label{4.3}
\begin{aligned}
&\triangle^{+}_{z_{0},1}=\triangle^{-}_{z_{0},-1}=\{k\in\{1,\cdots,N\}|z_{k}|<z_{0}\},\\
&\triangle^{-}_{z_{0},1}=\triangle^{+}_{z_{0},-1}=\{k\in\{1,\cdots,N\}|z_{k}|>z_{0}\},
\end{aligned}
\end{align}
where the subscript $\eta=\pm1$ is defined by $\eta=sgn(t)$.
\begin{align}\label{4.4}
I_{+}=(-\infty,-z_{0})\cup(z_{0},+\infty),~~I_{-}=[-z_{0},z_{0}].
\end{align}

In the following analysis, we mainly pay attention to the case that $t\rightarrow+\infty$, and the case $t\rightarrow-\infty$ can be analyzed in a similarly way.

In order to re-normalize the Riemann-Hilbert problem\eqref{RH-2}, we first introduce the following function
\begin{align*}
\delta(z)=\exp\left[i\int_{-z_{0}}^{z_{0}}\frac{\nu(s)}{s-z}ds\right],
~~\nu(s)=-\frac{1}{2\pi}\log(1+|r(s)|^{2}).
\end{align*}
and
\begin{align}\label{4.5}
T(z)=\prod_{k\in\triangle_{z_{0},1}^{+}}\frac{z-z^{*}_{k}}{z-z_{k}}\delta(z),
\end{align}
which has the following properties.
\begin{prop}\label{T-property} $T(z)$ admits that\\
($a$) $T$ is meromorphic in $C\setminus I_{-}$;\\
($b$) For $z\in C\setminus I_{-}$, $T^{*}(z^{*})=\frac{1}{T(z)}$;\\
($c$) For $z\in I_{-}$, $t\rightarrow+\infty$, the boundary values $T_{\pm}$ satisfy
\begin{align}\label{4.6}
T_{+}(z)/T_{-}(z)=1+|r(z)|^{2}, z\in I_{-};
\end{align}
($d$) As $|z|\rightarrow \infty $ with $|arg(z)|\leq c<\pi$,
\begin{align}\label{4.7}
T(z)=1+\frac{i}{z}\left(2\sum_{k\in\triangle_{z_{0},1}^{+}}Im~z_{k}-\int_{--z{0}}^{z_{0}}\nu(s)ds\right)+O(z^{-2});
\end{align}
($e$) As $z\rightarrow z_{0}$ along any ray $z_{0}+e^{i\phi}R_{+}$ with $|\phi|\leq c<\pi$
\begin{align}\label{4.8}
|T(z,z_{0})-T_{0}(\pm z_{0})(z\mp z_{0})^{i\nu(\pm z_{0})}|\leq c\parallel r\parallel_{H^{1}(R)}|z\mp z_{0}|^{\frac{1}{2}},
\end{align}
where $T_{0}(z_{0})$ is the complex unit
\begin{align}\label{4.9}
\begin{split}
&T_{0}(\pm z_{0})=\prod_{k\in\triangle_{z_{0},1}^{+}}\left(\frac{\pm z_{0}-z^{*}_{k}}
{\pm z_{0}-z_{k}}\right)e^{i\beta^{\pm}(z_{0},\pm z_{0})},\\
&\beta^{\pm}(z,\pm z_{0})=-\nu(\pm z_{0})\log(z\mp z_{0}+1)+\int_{-z_{0}}^{z_{0}}
\frac{\nu(s)-\chi_{\pm}(s)\nu(\pm z_{0})}{s-z}ds.
\end{split}
\end{align}
Here $\chi_{\pm}(s)=1$ are the characteristic functions  of the interval $s\in(z_{0}-1,z_{0})$ and $s\in(-z_{0},-z_{0}+1)$ respectively.\\
($f$) As $z\rightarrow0$, $T(z)$ can be expressed as
\begin{align}\label{4.10}
T(z)=T(0)(1+zT_{1})+O(z^{2}),
\end{align}
where $T_{1}=2\sum_{k\in\triangle_{z_{0},1}^{+}}\frac{Im~z_{k}}{z_{k}}-\int_{-z_{0}}^{z_{0}}\frac{\nu(s)}{s^{2}}ds$.
\end{prop}

\begin{proof}
The properties of $T(z)$ can be proved by a direct calculation, for details, see \cite{Yang-SP},\cite{Li-cgNLS}.
\end{proof}

Then, by applying the function $T(z)$, we introduce a transformation
\begin{align}\label{Trans-1}
M^{(1)}(y,t;z)=M(y,t;z)T(z)^{\sigma_{3}},
\end{align}
which admits the following matrix RHP.
\begin{RHP}\label{RH-3}
Find an analysis function $M^{(1)}$ with the following properties:
\begin{itemize}
  \item $M^{(1)}$ is meromorphic on $C\setminus R$;
  \item $[M^{(1)}(y,t;z^{*})]^{*}=\sigma_{2}M^{(1)}(x,t;z)\sigma_{2}$;
  \item $M^{(1)}(z)= \mathbb{I}+O(z^{-1})$ as $z\rightarrow \infty$;
  \item $M^{(1)}_{\pm}(z)$ satisfy the jump relationship $M^{(1)}_{+}(z)=M^{(1)}_{-}(z)V^{(1)}(z)$, where
      \begin{align}\label{4.11}
       V^{(1)}=\left\{\begin{aligned}
      \left(
        \begin{array}{cc}
      1 & 0 \\
      r^{*}(z)T(z)^{2}e^{-2it\theta} & 1 \\
        \end{array}
      \right)\left(
     \begin{array}{cc}
       1 & r(z)T(z)^{-2}e^{2it\theta} \\
       0 & 1 \\
      \end{array}
    \right),z\in\mathbb{R}\setminus I_{-},\\
   \left(
    \begin{array}{cc}
    1 & \frac{r(z)T_{-}(z)^{-2}}{1+|r(z)|^{2}}e^{2it\theta} \\
    0 & 1 \\
     \end{array}
   \right)\left(
    \begin{array}{cc}
    1 &  \\
    \frac{r^{*}(z)T_{+}(z)^{2}}{1+|r(z)|^{2}}e^{-2it\theta} & 1 \\
   \end{array}
  \right),z\in I_{-}\setminus\{\pm z_{0}\}.
   \end{aligned}\right.
   \end{align}
   \item $M^{(1)}(z)$ has simple poles at each $z_{k}\in \mathcal{Z}$ and $z^{*}_{k}\in \mathcal{Z}^{*}$ at which
\begin{align}\label{4.12}
\begin{split}
&\mathop{Res}\limits_{z=z_{k}}M^{(1)}=\left\{\begin{aligned}
&\lim_{z\rightarrow z_{k}}M^{(1)}\left(\begin{array}{cc}
    0 & 0\\
    c_{k}^{-1}\left((\frac{1}{T})'(z_{k})\right)^{-2}e^{-2it\theta} & 0 \\
  \end{array}
\right),k\in \triangle_{z_{0},1}^{+}\\
&\lim_{z\rightarrow z_{k}}M^{(1)}\left(
  \begin{array}{cc}
    0 & c_{k}T^{-2}(z_{k})e^{2it\theta}\\
     & 0 \\
  \end{array}\right),k\in \triangle_{z_{0},1}^{-}
\end{aligned}\right.\\
&\mathop{Res}\limits_{z=z^{*}_{k}}M^{(1)}=\left\{\begin{aligned}
&\lim_{z\rightarrow z^{*}_{k}}M^{(1)}\left(\begin{array}{cc}
    0 & 0\\
    -(c^{*}_{k})^{-1}(T'(z^{*}_{k}))^{-2}e^{-2it\theta} & 0 \\
  \end{array}
\right),k\in \triangle_{z_{0},1}^{-}\\
&\lim_{z\rightarrow z^{*}_{k}}M^{(1)}\left(
  \begin{array}{cc}
    0 & -c^{*}_{k}(T(z^{*}_{k}))^{2}e^{2it\theta} \\
    0 & 0 \\
  \end{array}\right),k\in \triangle_{z_{0},1}^{+}
\end{aligned}\right.
\end{split}
\end{align}
\end{itemize}
\end{RHP}

\begin{proof}
Based on the above analysis, it is easy to prove the analyticity, jump  conditions, asymptotic behaviors and residue condition, for detail, see \cite{Yang-SP},\cite{Li-cgNLS}.
\end{proof}

\section{Continuous extension to a mixed $\bar{\partial}$-RH problem}
In this section, our purpose is to extend the jump matrix off the real axis. Here we just need the extension is continuous, and the oscillation term along the new contours are decaying. Firstly, we introduce some the contours
\begin{align}\label{5.11}
\begin{split}
&\Sigma_{j}=z_{0}+e^{i(2j-1)\pi/4}R_{+},~~j=1,4;\\
&\Sigma_{j}=z_{0}+e^{i(2j-1)\pi/4}h,~~h\in\left(0,\frac{\sqrt{2}}{2}z_{0}\right),~~j=2,3;\\
&\Sigma_{j}=-z_{0}+e^{i(2j-1)\pi/4}h,~~h\in\left(0,\frac{\sqrt{2}}{2}z_{0}\right),~~j=5,8;\\
&\Sigma_{j}=-z_{0}+e^{i(2j-1)\pi/4}R_{+},~~j=6,7;\\
&\Sigma_{j}=e^{i(2j-1)\pi/4}h,~~h\in\left(0,\frac{\sqrt{2}}{2}z_{0}\right),~~,j=9,10,11,12;\\
&\Sigma^{2}=\cup_{j=1}^{12}\Sigma_{j}.
\end{split}
\end{align}
Then, the complex plane $\mathbb{C}$ is separated into ten sectors which are denoted by $\Omega_{j}(j=1,2,\ldots,10)$ respectively, and shown in Figure 2.\\

\centerline{\begin{tikzpicture}[scale=0.8]
\path [fill=pink] (-4,0) -- (-2,2) to (0,0) -- (-2,-2);
\path [fill=pink] (4,0) -- (2,2) to (0,0) -- (2,-2);
\path [fill=pink] (4,0) -- (7,3) to (8,3) -- (8,0);
\path [fill=pink] (4,0) -- (7,-3) to (8,-3) -- (8,0);
\path [fill=pink] (-4,0) -- (-7,3) to (-8,3) -- (-8,0);
\path [fill=pink] (-4,0) -- (-7,-3) to (-8,-3) -- (-8,0);
\draw [dashed](-8,0)--(8,0);
\draw[->][thick](4,0)--(6,2);
\draw[-][thick](6,2)--(7,3);
\draw[->][thick](4,0)--(6,-2);
\draw[-][thick](6,-2)--(7,-3);
\draw[->][thick](4,0)--(3,1);
\draw[-][thick](3,1)--(2,2);
\draw[->][thick](4,0)--(3,-1);
\draw[-][thick](3,-1)--(2,-2);
\draw[->][thick](2,2)--(1,1);
\draw[->][thick](1,1)--(-1,-1);
\draw[->][thick](2,-2)--(1,-1);
\draw[->][thick](1,-1)--(-1,1);
\draw[-][thick](-1,1)--(-2,2);
\draw[->][thick](-2,2)--(-3,1);
\draw[-][thick](-3,1)--(-6,-2);
\draw[->][thick](-2,2)--(-3,1);
\draw[->][thick](-7,3)--(-6,2);
\draw[-][thick](-1,-1)--(-2,-2);
\draw[->][thick](-2,-2)--(-3,-1);
\draw[-][thick](-3,-1)--(-6,2);
\draw[->][thick](-7,-3)--(-6,-2);
\draw[fill] (0,0)node[below]{$0$} circle [radius=0.08];
\draw[fill] (4,0)node[below]{$z_{0}$} circle [radius=0.08];
\draw[fill] (-4,0)node[below]{$-z_{0}$} circle [radius=0.08];
\draw[fill] (2,0)node[below]{$\Omega_{4}$};
\draw[fill] (2,0)node[above]{$\Omega_{3}$};
\draw[fill] (-2,0)node[below]{$\Omega_{9}$};
\draw[fill] (-2,0)node[above]{$\Omega_{8}$};
\draw[fill] (6,0)node[below]{$\Omega_{6}$};
\draw[fill] (6,0)node[above]{$\Omega_{1}$};
\draw[fill] (-6,0)node[below]{$\Omega_{10}$};
\draw[fill] (-6,0)node[above]{$\Omega_{7}$};
\draw[fill] (0,2)node[above]{$\Omega_{2}$};
\draw[fill] (0,-2)node[below]{$\Omega_{5}$};
\draw[fill] (7,3)node[left]{$\Sigma_{1}$};
\draw[fill] (7,-3)node[left]{$\Sigma_{4}$};
\draw[fill] (-7,3)node[right]{$\Sigma_{6}$};
\draw[fill] (-7,-3)node[right]{$\Sigma_{7}$};
\draw[fill] (-2,2)node[left]{$\Sigma_{5}$};
\draw[fill] (-2,2)node[right]{$\Sigma_{10}$};
\draw[fill] (-2,-2)node[left]{$\Sigma_{8}$};
\draw[fill] (-2,-2)node[right]{$\Sigma_{11}$};
\draw[fill] (2,2)node[left]{$\Sigma_{9}$};
\draw[fill] (2,2)node[right]{$\Sigma_{2}$};
\draw[fill] (2,-2)node[left]{$\Sigma_{12}$};
\draw[fill] (2,-2)node[right]{$\Sigma_{3}$};
\draw[fill] (0,3)node[above]{$sgn(t)=1(t\rightarrow+\infty)$};
\end{tikzpicture}}
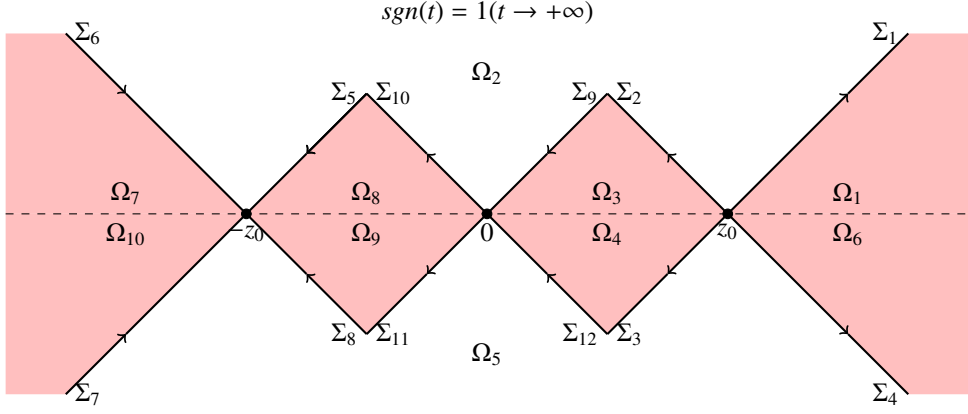
\centerline{\noindent {\small \textbf{Figure 2.}  Definition of $R^{(2)}$ in different domains.}}

Moreover, define
\begin{align}
\rho=\frac{1}{2}\min_{(z_{a}\neq z_{b})\in \mathcal{Z}\cup \mathcal{Z}^{*}}\{|z_{a}-z_{b}|\},
\end{align}
and $\chi_{Z} \in C_{0}^{\infty} (C, [0, 1])$ which is supported near the discrete spectrum $\mathcal{Z}\cup \mathcal{Z}^{*}$ such that
\begin{align}\label{5.2}
\chi_{Z}(z)=\left\{\begin{aligned}
&1,~~dist(z,\mathcal{Z}\cup \mathcal{Z}^{*})<\rho/3, \\
&0,~~dist(z,\mathcal{Z}\cup \mathcal{Z}^{*})>2\rho/3.
\end{aligned}\right.
\end{align}
Also we can verify that $dist(\mathcal{Z}\cup \mathcal{Z}^{*}, R)>\rho, k=1,2,\ldots,N.$

Next, in order to achieve the purpose of extending the jump matrix onto the new contours along which oscillation term are decaying, we introduce a transformation
\begin{align}\label{Trans-2}
M^{(2)}=M^{(1)}R^{(2)},
\end{align}
where $R^{(2)}$ possesses some restrictions.
\begin{itemize}
  \item The aim of the transformation is to extend the jump matrix onto the new contours $\Sigma^{(2)}$. So on the real axis, $M^{(2)}$ must have no jump.
  \item To guarantee that the $\bar{\partial}$-contribution  has little impact on the large-time asymptotic solution of $u(x,t)$, the norm of $R^{(2)}$ need to be controlled.
  \item The introduced transformation need to have no impact on the residue condition.
\end{itemize}

Then, we define $R^{(2)}$ as
\begin{align}
R^{(2)}=\left\{\begin{aligned}
&\left(
  \begin{array}{cc}
    1 & (-1)^{m_{j}}R_{j}e^{2it\theta}  \\
    0 & 1 \\
  \end{array}
\right), ~~&z\in\Omega_{j},~~j=1,4,7,9,\\
&\left(
  \begin{array}{cc}
    1 & 0 \\
    (-1)^{m_{j}}R_{j}e^{-2it\theta} & 1 \\
  \end{array}
\right), ~~&z\in\Omega_{j},~~j=3,6,8,10,\\
&\left(
  \begin{array}{cc}
    1 & 0 \\
    0 & 1 \\
  \end{array}
\right),~~ &z\in\Omega_{2}\cup\Omega_{5},
\end{aligned}
\right.
\end{align}
where $m_{j}=1(j=1,3,7,8)$ and $m_{j}=0(j=4,6,9,10)$ and $R_{j}(z)$ are defined in the following proposition.
\begin{prop}\label{R-property}
There exists functions $R_{j}: \Omega_{j} \rightarrow C, j= 1, 3, 4, 6,7,8,9,10$ such that
\begin{align*}
&R_{1}(z)=\left\{\begin{aligned}&r(z)T^{-2}(z), ~~~~z\in(z_{0}, \infty),\\
&f_{1}=r(z_{0})T_{0}^{-2}(z_{0})(z-z_{0})^{-2i\nu(z_{0})}(1-\chi_{Z}(z)), z\in\Sigma_{1},
\end{aligned}\right.\\
&R_{3}(z)=\left\{\begin{aligned}&\frac{r^{*}(z)}{1+|r(z)|^{2}}T_{+}^{2}(z), ~~~~z\in(0, z_{0}),\\
&f_{3}=\frac{r^{*}(z_{0})}{1+|r(z_{0})|^{2}}T_{0}^{2}(z_{0}) (z-z_{0})^{2i\nu(z_{0})}(1-\chi_{Z}(z)), z\in\Sigma_{2},
\end{aligned}\right.\\
&R_{4}(z)=\left\{\begin{aligned}&\frac{r(z)}{1+|r(z)|^{2}}T_{-}^{2}(z), ~~~~z\in(0, z_{0}),\\
&f_{4}=\frac{r(z_{0})}{1+|r(z_{0})|^{2}}T_{0}^{-2}(z_{0}) (z-z_{0})^{-2i\nu(z_{0})}(1-\chi_{Z}(z)), z\in\Sigma_{3},
\end{aligned}\right.\\
&R_{6}(z)=\left\{\begin{aligned}&r^{*}(z)T^{2}(z), ~~~~z\in(z_{0}, \infty),\\
&f_{6}=r^{*}(z_{0})T_{0}^{2}(z_{0})(z-z_{0})^{2i\nu(z_{0})}(1-\chi_{Z}(z)), z\in\Sigma_{4}.
\end{aligned}\right.\\
\end{align*}
\begin{align*}
&R_{7}(z)=\left\{\begin{aligned}&r(z)T^{-2}(z), ~~~~z\in(-\infty, z_{0}),\\
&f_{7}=r(-z_{0})T_{0}^{-2}(-z_{0})(z+z_{0})^{-2i\nu(-z_{0})}(1-\chi_{Z}(z)), z\in\Sigma_{5},
\end{aligned}\right.\\
&R_{8}(z)=\left\{\begin{aligned}&\frac{r^{*}(z)}{1+|r(z)|^{2}}T_{+}^{2}(z), ~~~~z\in(-z_{0},0),\\
&f_{8}=\frac{r^{*}(-z_{0})}{1+|r(-z_{0})|^{2}}T_{0}^{2}(-z_{0}) (z+z_{0})^{2i\nu(-z_{0})}(1-\chi_{Z}(z)), z\in\Sigma_{5},
\end{aligned}\right.\\
&R_{9}(z)=\left\{\begin{aligned}&\frac{r(z)}{1+|r(z)|^{2}}T_{-}^{2}(z), ~~~~z\in(-z_{0},0),\\
&f_{9}=\frac{r(-z_{0})}{1+|r(-z_{0})|^{2}}T_{0}^{-2}(-z_{0}) (z+z_{0})^{-2i\nu(-z_{0})}(1-\chi_{Z}(z)), z\in\Sigma_{8},
\end{aligned}\right.\\
&R_{10}(z)=\left\{\begin{aligned}&r^{*}(z)T^{2}(z), ~~~~z\in(-\infty, -z_{0}),\\
&f_{10}=r^{*}(-z_{0})T_{0}^{2}(-z_{0})(z+z_{0})^{2i\nu(-z_{0})}(1-\chi_{Z}(z)), z\in\Sigma_{7}.
\end{aligned}\right.
\end{align*}
And $R_{j}$ admit that
\begin{align}
j=1,3,4,6~~~~\left\{\begin{aligned}&|R_{j}(z)|\leq c_{1}\sin^{2}(\arg (z-z_{0}))+c_{2}\left<Rez\right>^{-1/2},\\
&|\bar{\partial}R_{j}(z)|\leq c_{1}\bar{\partial}\chi_{Z}(z)+c_{2}|z-z_{0}|^{-1/2}+c_{3}|p'_{j}(Rez)|,
\end{aligned}\right.\label{R-estimate-1}\\
j=7,8,9,10~~~~\left\{\begin{aligned}&|R_{j}(z)|\leq c_{1}\sin^{2}(\arg (z+z_{0}))+c_{2}\left<Rez\right>^{-1/2},\\
&|\bar{\partial}R_{j}(z)|\leq c_{1}\bar{\partial}\chi_{Z}(z)+c_{2}|z+z_{0}|^{-1/2}+c_{3}|p'_{j}(Rez)|,
\end{aligned}\right.\label{R-estimate-2}\\
\bar{\partial}R_{j}(z)=0,z\in \Omega_{2}\cup\Omega_{5},or~ dist(z,\mathcal{Z}\cup\mathcal{Z}^{*})\leq \rho/3,
\end{align}
where
\begin{gather*}
\left<Rez\right>=\sqrt{1+(Rez)^{2}},\\
p_{1}=p_{7}=r(z),~~p_{3}=p_{8}=\frac{r(z)}{1+|r(z)|^{2}},\\
p_{6}=p_{10}=r^{*}(z),~~p_{4}=p_{9}=\frac{r^{*}(z)}{1+|r(z)|^{2}}.
\end{gather*}
\end{prop}

The proof process of the results in Proposition \ref{R-property} is similar to that in \cite{AIHP,Li-cgNLS}.

Then, based on $R^{(2)}$ shown in Proposition \ref{R-property} and applying the transformation\eqref{Trans-2}, we obtain $M^{(2)}$ which admits the following mixed $\bar{\partial}$-RH problem.

\begin{RHP}\label{RH-4}
Find a matrix value function $M^{(2)}$, admitting
\begin{itemize}
 \item $M^{(2)}(x,t,z)$ is continuous in $\mathbb{C}\setminus(\Sigma^{(2)}\cup\mathcal{Z}\cup\mathcal{Z}^{*})$.
 \item  $[M^{(2)}(y,t;z^{*})]^{*}=\sigma_{2}M^{(2)}(x,t;z)\sigma_{2}$.
 \item $M_+^{(2)}(x,t,z)=M_{-}^{(2)}(x,t,z)V^{(2)}(x,t,z),$ ~~ $z\in\Sigma^{(2)}$, where the jump matrix $V^{(2)}(x,t,z)$ satisfies
 \begin{align}
V^{(2)}=\left\{\begin{aligned}
&\left(
  \begin{array}{cc}
    1 & R_{1}e^{2it\theta}  \\
    0 & 1 \\
  \end{array}
\right), ~~&z\in\Sigma_{1},\\
&\left(
  \begin{array}{cc}
    1 & 0 \\
    R_{3}e^{-2it\theta} & 1 \\
  \end{array}
\right), ~~&z\in\Sigma_{2}\cup\Sigma_{9},\\
&\left(
  \begin{array}{cc}
    1 & R_{4}e^{2it\theta}  \\
    0 & 1 \\
  \end{array}
\right), ~~&z\in\Sigma_{3}\cup\Sigma_{12},\\
&\left(
  \begin{array}{cc}
    1 & 0 \\
    R_{6}e^{-2it\theta} & 1 \\
  \end{array}
\right), ~~&z\in\Sigma_{4},\\
&\left(
  \begin{array}{cc}
    1 & 0 \\
    R_{8}e^{-2it\theta} & 1 \\
  \end{array}
\right), ~~&z\in\Sigma_{5}\cup\Sigma_{10},\\
&\left(
  \begin{array}{cc}
    1 & R_{7}e^{2it\theta}  \\
    0 & 1 \\
  \end{array}
\right), ~~&z\in\Sigma_{6},\\
&\left(
  \begin{array}{cc}
    1 & 0 \\
    R_{10}e^{-2it\theta} & 1 \\
  \end{array}
\right), ~~&z\in\Sigma_{7},\\
&\left(
  \begin{array}{cc}
    1 & R_{9}e^{2it\theta}  \\
    0 & 1 \\
  \end{array}
\right), ~~&z\in\Sigma_{8}\cup\Sigma_{11};\\
\end{aligned}
\right.
\end{align}
\item $M^{(2)}(x,t,z)\rightarrow \mathbb{I},$ \quad $z\rightarrow\infty$.
\item For $\mathbb{C}\setminus(\Sigma^{(2)}\cup\mathcal{Z}\cup\mathcal{Z}^{*})$, $\bar{\partial}M^{(2)}=M^{(2)}\bar{\partial}\mathcal{R}^{(2)}(z),$ where
   \begin{align}\label{dbar-R2}
\bar{\partial}R^{(2)}=\left\{\begin{aligned}
&\left(
  \begin{array}{cc}
    1 & (-1)^{m_{j}}\bar{\partial}R_{j}e^{2it\theta}  \\
    0 & 1 \\
  \end{array}
\right), ~~&z\in\Omega_{j},~~j=1,4,7,9,\\
&\left(
  \begin{array}{cc}
    1 & 0 \\
    (-1)^{m_{j}}\bar{\partial}R_{j}e^{-2it\theta} & 1 \\
  \end{array}
\right), ~~&z\in\Omega_{j},~~j=3,6,8,10,\\
&\left(
  \begin{array}{cc}
    0 & 0 \\
    0 & 0 \\
  \end{array}
\right),~~ &z\in\Omega_{2}\cup\Omega_{5},
\end{aligned}
\right.
\end{align}
  \item  $M^{(2)}$ admits the residue conditions at poles $z_{k} \in \mathcal{Z}$ and $z^{*}_{k} \in \mathcal{Z}^{*}$, i.e.,
     \begin{align}
\begin{split}
&\mathop{Res}\limits_{z=z_{k}}M^{(2)}=\left\{\begin{aligned}
&\lim_{z\rightarrow z_{k}}M^{(1)}\left(\begin{array}{cc}
    0 & 0\\
    c_{k}^{-1}\left((\frac{1}{T})'(z_{k})\right)^{-2}e^{-2it\theta} & 0 \\
  \end{array}
\right),k\in \triangle_{z_{0},1}^{+}\\
&\lim_{z\rightarrow z_{k}}M^{(2)}\left(
  \begin{array}{cc}
    0 & c_{k}T^{-2}(z_{k})e^{2it\theta}\\
     & 0 \\
  \end{array}\right),k\in \triangle_{z_{0},1}^{-}
\end{aligned}\right.\\
&\mathop{Res}\limits_{z=z^{*}_{k}}M^{(1)}=\left\{\begin{aligned}
&\lim_{z\rightarrow z^{*}_{k}}M^{(1)}\left(\begin{array}{cc}
    0 & 0\\
    -(c^{*}_{k})^{-1}(T'(z^{*}_{k}))^{-2}e^{-2it\theta} & 0 \\
  \end{array}
\right),k\in \triangle_{z_{0},1}^{-}\\
&\lim_{z\rightarrow z^{*}_{k}}M^{(1)}\left(
  \begin{array}{cc}
    0 & -c^{*}_{k}(T(z^{*}_{k}))^{2}e^{2it\theta} \\
    0 & 0 \\
  \end{array}\right),k\in \triangle_{z_{0},1}^{+}
\end{aligned}\right.
\end{split}
\end{align}
\end{itemize}
\end{RHP}

\section{Decomposition of the mixed $\bar{\partial}$-RH problem}
The purpose of this section is to decompose the mixed $\bar{\partial}$-RH problem into two parts which include a model RH problem with $\bar{\partial}R^{(2)}=0$ and a pure $\bar{\partial}$-RH problem with $\bar{\partial}R^{(2)}\neq0$. We denote $M^{(2)}_{RHP}$ as the solution of the model RH problem, and first construct a RH problem for $M^{(2)}_{RHP}$.

\begin{RHP}\label{RH-rhp}
Find a matrix value function $M^{(2)}_{RHP}$, admitting
\begin{itemize}
 \item $M^{(2)}_{RHP}$ is analytical in $\mathbb{C}\backslash(\Sigma^{(2)}\cup\mathcal{Z}\cup\mathcal{Z}^{*})$;
 \item  $[M^{(2)}_{RHP}(y,t;z^{*})]^{*}=\sigma_{2}M^{(2)}_{RHP}(x,t;z)\sigma_{2}$;
 \item $M^{(2)}_{RHP,+}(x,t,z)=M^{(2)}_{RHP,-}(x,t,z)V^{(2)}(x,t,z),$ \quad $z\in\Sigma^{(2)}$, where $V^{(2)}(x,t,z)$ is the same with the jump matrix appears in RHP \ref{RH-3};
 \item As $z\rightarrow\infty$, $M^{(2)}_{RHP}(x,t,z)=\mathbb{I}+o(z^{-1})$;
 \item $M^{(2)}_{RHP}$ possesses the same residue condition with $M^{(2)}$.
 \end{itemize}
\end{RHP}

Then, if we can prove the existence of the solution of $M^{(2)}_{RHP}$, the RHP \ref{RH-4} can be reduced to a pure $\bar{\partial}$-RH problem. The existence of the solution of $M^{(2)}_{RHP}$ will be proved in section \ref{section-pure RH problem}. Now, supposing that the solution $M^{(2)}_{RHP}$ exists, and constructing a transformation
\begin{align}\label{delate-pure-RHP}
M^{(3)}(z)=M^{(2)}(z)M^{(2)}_{RHP}(z)^{-1},
\end{align}
we obtain the following pure $\bar{\partial}$-RH problem.
\begin{RHP}\label{RH-5}
Find a matrix value function $M^{(3)}$, admitting
\begin{itemize}
 \item $M^{(3)}$ is continuous with sectionally continuous first partial derivatives in $\mathbb{C}\backslash(\Sigma^{(2)}\cup\mathcal{Z}\cup\mathcal{Z}^{*})$;
 \item  $[M^{(3)}(y,t;z^{*})]^{*}=\sigma_{2}M^{(3)}(x,t;z)\sigma_{2}$;
 \item For $z\in \mathbb{C}$, we obtain $\bar{\partial}M^{(3)}(z)=M^{(3)}(z)W^{(3)}(z)$,
       where
       \begin{align}\label{5.1}
       W^{(3)}=M_{RHP}^{(2)}(z)\bar{\partial}R^{(2)}M_{RHP}^{(2)}(z)^{-1};
       \end{align}
 \item As $z\rightarrow\infty$, ~~$M^{(3)}(z)=I+o(z^{-1})$.
 \end{itemize}
\end{RHP}
\begin{proof}
According to the properties of the $M^{(2)}_{RHP}$ and $M^{(2)}$ for RHP \ref{RH-rhp} and RHP \ref{RH-4}, the analyticity and asymptotic properties of $M^{(3)}$ can be derived easily. Noting the fact that $M^{(2)}_{RHP}$ possesses the same jump matrix with $M^{(2)}$, we obtain that
\begin{align*}
M^{(3)}_{-}(z)^{-1}M^{(3)}_{+}(z)&=M^{(2)}_{RHP,-}(z)M^{(2)}_{-}(z)^{-1}M^{(2)}_{+}(z)M^{(2)}_{RHP,+}(z)^{-1}\\
&=M^{(2)}_{RHP,-}(z)V^{2}(z)(M^{(2)}_{RHP,-}(z)V^{2}(z))^{-1}=\mathbb{I},
\end{align*}
which implies that $M^{(3)}$ has no jump. Also, it is easy to prove that there exists no pole in $M^{(3)}$ by a simple analysis, for details, see \cite{AIHP, Yang-SP,Li-cgNLS}.
\end{proof}

\section{The pure RH problem}\label{section-pure RH problem}

In this section, we construct the solution $M^{(2)}_{RHP}$ of RHP \ref{RH-rhp}. Define that
\begin{align*}
\mathcal{U}_{\pm z_0}=\{z:|z\mp z_0|<min\{\frac{z_{0}}{2},\rho/3\}\}.
\end{align*}
Then, we can decompose $M^{(2)}_{RHP}$ into two parts
\begin{align}\label{Mrhp}
M^{(2)}_{RHP}(z)=\left\{\begin{aligned}
&E(z)M^{(out)}(z), &&z\in\mathbb{C}\setminus\mathcal{U}_{\pm z_0},\\
&E(z)M^{(\pm z_0)}(z), &&z\in \mathcal{U}_{\pm z_0},
\end{aligned} \right.
\end{align}
from which we obtain that $M^{\pm z_0}(z)$ possesses no poles in $\mathcal{U}_{\pm z_0}$.
Besides, $M^{(out)}$ solves a model RHP, the solution of $M^{(\pm z_0)}$ can be approximated with a known parabolic cylinder model in $\mathcal{U}_{\pm z_0}$, and $E(z)$ is an error function which is a solution of a small-norm Riemann-Hilbert problem.

Additionally, for the jump matrix $V^{(2)}$, we evaluate its estimate.
\begin{align}
&||V^{(2)}-\mathbb{I}||_{L^{\infty}(\Sigma_{\pm}^{(2)}\setminus\mathcal{U}_{\pm z_0})}
=O\left(e^{-\frac{\sqrt{2}}{16}t|z\mp z_0|^{2}}\right),\label{V2-Est-1}\\
&||V^{(2)}-\mathbb{I}||_{L^{\infty}(\Sigma_{0}^{(2)}}
=O\left(e^{-\frac{t}{4z_{0}}}\right),\label{V2-Est-2}
\end{align}
where $\Sigma_{\pm}^{(2)}$ and $\Sigma_{0}^{(2)}$ are defined as
\begin{align*}
\Sigma_{+}^{(2)}=\cup_{j=1}^{4}\Sigma_{j},~~\Sigma_{-}^{(2)}=\cup_{j=5}^{8}\Sigma_{j},
~~\Sigma_{0}^{(2)}=\cup_{j=9}^{12}\Sigma_{j}.
\end{align*}

According to the above estimate of the jump matrix $V^{(2)}$, we know that if we omit the jump condition of  $M^{(2)}_{RHP}(z)$, there only exists exponentially small error with respect to $t$ outside the $\mathcal{U}_{+z_0}\cup\mathcal{U}_{-z_0}$. In addition, noting the fact that
$V^{(2)}\rightarrow I$ as $z\rightarrow 0$, it is not necessary to study the  neighborhood of $z=0$ alone.

\subsection{Outer model RH problem: $M^{(out)}$}
In this section, we establish a model RH problem and prove that its solution can be approximated by finite sum of soliton solutions.

\begin{RHP}\label{RH-6}
Find a matrix value function $M^{(out)}(y,t;z)$, admitting
\begin{itemize}
  \item $M^{(out)}(y,t;z)$ is analytical in $\mathbb{C}\setminus(\Sigma^{(2)}\cup\mathcal{Z}\cup\mathcal{Z}^{*})$;
  \item $[M^{(out)}(y,t;z^{*})]^{*}=\sigma_{2}M^{(out)}(y,t;z)\sigma_{2}$;
  \item As $z\rightarrow\infty$,
       \begin{align}
       M^{(out)}(y,t;z)=\mathbb{I}+o(z^{-1});
       \end{align}
  \item $M^{(out)}(y,t;z)$ has simple poles at each point in $\mathcal{Z}\cup\mathcal{Z}^{*}$ admitting the same residue condition in RHP \ref{RH-4} with $M^{(out)}(y,t;z)$ replacing $M^{(2)}(y,t;z)$.
\end{itemize}
\end{RHP}

Before we investigate the solution of $M^{(out)}(x,t;z)$ for RHP \ref{RH-6}, we first study RHP \ref{RH-2} for the case of reflectionless. Under this condition, $M(y,t;z)$ has no jump, and we obtain the following Riemann-Hilbert problem from RHP \ref{RH-2}.
\begin{RHP}\label{RH-7}
Find a matrix value function $M(x,t;z|\sigma_{d})$, admitting
\begin{itemize}
  \item $M(y,t;z|\sigma_{d})$ is analytical in $\mathbb{C}\setminus(\mathcal{Z}\cup\mathcal{Z}^{*})$;
  \item $M^{*}(y,t;z^{*}|\sigma_{d})=\sigma_{2}M(y,t;z|\sigma_{d})\sigma_{2}$;
  \item $M(y,t;z|\sigma_{d})=\mathbb{I}+O(z^{-1})$, \quad $z\rightarrow\infty$;
  \item $M(y,t;z|\sigma_{d})$ satisfies the following residue conditions at simple poles $z_{k}\in\mathcal{Z}$ and $z_{k}^{*}\in\mathcal{Z}^{*}$
\begin{align}
\begin{aligned}
&\mathop{Res}_{z=z_{k}}M(x,t;z|\sigma_{d})=\mathop{lim}_{z\rightarrow z_{k}}M(x,t;z|\sigma_{d})N_{k},\\
&\mathop{Res}_{z=z_{k}^{*}}M(x,t;z|\sigma_{d})=\mathop{lim}_{z\rightarrow z_{k}^{*}}M(x,t;z|\sigma_{d})\sigma_{2}N^{*}_{k}\sigma_{2},
\end{aligned}
\end{align}
where $\sigma_{d}=\{(z_{k}, c_{k}), z_{k}\in\mathcal{Z}\}^{N}_{k=1}$, which satisfy $z_{k}\neq z_{j}$ for $k\neq j$, are scattering data , and
\begin{gather*}
N_{k}=\left(\begin{aligned}
\begin{array}{cc}
  0 & \gamma_{k}(x,t) \\
  0 & 0
\end{array}
\end{aligned}\right),~
\gamma_{k}(x,t)=c_{k}e^{2it\theta(z_{k})},\\
\theta(z_{k})=\frac{z_{k}}{4}\left(\frac{1}{z_{0}^{2}}+\frac{1}{z_{k}^{2}}\right).
\end{gather*}
\end{itemize}
\end{RHP}
\begin{prop}
The RHP \ref{RH-7} exists unique solution. Additionally, the solution admits
\begin{align}\label{6.1}
\|M(x,t;z|\sigma_{d})\|_{L^{\infty}(\mathbb{C}\setminus(\mathcal{Z}\cup\mathcal{Z}^{*}))}\lesssim 1.
\end{align}
\end{prop}
\begin{proof}
According to the Liouville's theorem, the uniqueness of the solution is obvious. The existence of RHP \ref{RH-7} and Eq.\eqref{6.1} can be proved by simple calculation which is similar to the literature \cite{Yang-SP,AIHP}.
\end{proof}

\subsubsection{Renormalization of the RHP for reflectionless case}
Under the reflectionless condition, recall that
\begin{align}
s_{22}(z)=\prod_{k=1}^{N}\left(\frac{z-z_{k}}{z-z^{*}_{k}}\right).
\end{align}
Taking $\triangle\subseteq\{1,2,\cdots,N\}$, $\triangledown\subseteq\{1,2,\cdots,N\}\setminus\triangle$, and defining
\begin{align}
s_{22}^{\triangle}=\prod_{k\in\triangle}\frac{z-z_{k}}{z-z^{*}_{k}},\quad
s_{22}^{\triangledown}=\frac{s_{22}}{s_{22}^{\triangle}}=
\prod_{k\in\triangledown}\frac{z-z_{k}}{z-z^{*}_{k}}.
\end{align}
Then, we introduce the normalization transformation
\begin{align}
M^{\triangle}(y,t;z|\sigma_{d}^{\triangle})=M(y,t;z|\sigma_{d})s_{22}^{\triangle}(z)^{-\sigma_{3}},
\end{align}
which splits the poles between the columns of $M(x,t;z|\sigma_{d})$ by selecting different $\triangle$. The scattering data $\sigma_{d}^{\triangle}$ are defined by $\sigma_{d}^{\triangle}=\{(z_{k}, c_{k}s_{22}^{\triangle}(z)^{2}), z_{k}\in\mathcal{Z}\}^{N}_{k=1}$  Then, we can get the modified Riemann-Hilbert problem.
\begin{RHP}\label{RH-8}
Given scattering data $\sigma^{\triangle}_{d}$ and $\triangle\subseteq\{1,2, \cdots,N\}$.
Find a matrix value function $M^{\triangle}$, admitting
\begin{itemize}
  \item $M^{\triangle}(y,t;z|\sigma^{\triangle}_{d})$ is analytical in $\mathbb{C}\setminus(\mathcal{Z}\bigcup\mathcal{Z}^{*})$;
  \item   $[M^{\triangle}(y,t;z^{*}|\sigma_{d}^{\triangle})]^{*}=\sigma_{2}M^{\triangle}(y,t;z|\sigma^{\triangle}_{d})\sigma_{2}$;
  \item $M^{\triangle}(y,t;z|\sigma^{\triangle}_{d})=\mathbb{I}+O(z^{-1})$, \quad $z\rightarrow\infty$;
  \item $M^{\triangle}(y,t;z|\sigma^{\triangle}_{d})$ satisfies the following residue conditions at simple poles $z_{k}\in\mathcal{Z}$ and $z_{k}^{*}\in\mathcal{Z}^{*}$
\begin{align}
\begin{aligned}
&\mathop{Res}_{z=z_{k}}M^{\triangle}(x,t;z|\sigma^{\triangle}_{d})=\mathop{lim}_{z\rightarrow z_{k}}M^{\vartriangle}(x,t;z|\sigma^{\triangle}_{d})N^{\triangle}_{k},\\
&\mathop{Res}_{z=z_{k}^{*}}M^{\triangle}(x,t;z|\sigma^{\triangle}_{d})=\mathop{lim}_{z\rightarrow z_{k}^{*}}M^{\triangle}(x,t;z|\sigma^{\triangle}_{d})\sigma_{2}(N^{\triangle}_{k})^{*}\sigma_{2},
\end{aligned}
\end{align}
where
\begin{gather}
N_{k}^{\triangle}=\left\{
                                   \begin{aligned}
\left(
  \begin{array}{cc}
    0 & \gamma_{k}^{\triangle} \\
    0 & 0 \\
  \end{array}
\right),\quad k\notin \triangle,\\
\left(
  \begin{array}{cc}
    0 & 0 \\
    \gamma_{k}^{\triangle} & 0 \\
  \end{array}
\right),\quad k\in \triangle,
\end{aligned}\right.~~\gamma_{k}^{\triangle}=\left\{
                                   \begin{aligned}
&c_{k}(s_{22}^{\triangle}(z_{k}))^{2}e^{2it\theta(z_{k})}\quad k\notin \triangle,\\
&c_{k}^{-1}(s_{22}^{\triangle'}(z_{k}))^{-2}e^{-2it\theta(z_{k})}\quad k\in \triangle.
\end{aligned}\right.
\end{gather}
\end{itemize}
\end{RHP}
Then, taking $\triangle=\triangle_{z_{0},1}^{+}$ and using $\sigma^{out}_{d}=\{(z_{k}, c_{k}\delta(z_{k})^{2}), z_{k}\in\mathcal{Z}\}^{N}_{k=1}$ instead of the scattering data $\sigma^{\triangle}_{d}$, we obtain that
\begin{gather}
M^{(out)}(z)=M^{\triangle_{z_{0},1}^{+}}(z)\delta(z)^{\sigma_{3}}
=M^{\vartriangle_{z_{0}}^{-}}(z|\sigma_{d}^{out}).
\end{gather}
From the above analysis, we note that $M^{\triangle}(y,t;z|\sigma^{out}_{d})$ is directly transformed from $M(y,t;z|\sigma_{d})$ which leads to that RHP \ref{RH-7} has unique solution.

For given scattering data $\sigma^{\triangle}_{d}$, the unique $N$-soliton solution of RHP \ref{RH-7} can be expressed as
\begin{align}
u_{sol}(y,t;\sigma^{\triangle}_{d})=\mathop{lim}_{z\rightarrow 0}\frac{\left(M(0;y,t|\sigma^{\triangle}_{d})^{-1}M(z;y,t|\sigma^{\triangle}_{d})\right)_{12}}{iz}.
\end{align}
This indicates that each normalization encodes $u_{sol}(y,t)$ in the same way. By selecting appropriate $\triangle$, the asymptotic limits in which $t\rightarrow\infty$ with $\frac{y}{t}$ bounded are under better asymptotic control. Next, we study the asymptotic behavior of the soliton solutions.

\subsubsection{Long-time behavior of soliton solutions}
We first define some notations
\begin{gather*}
I=\left\{z:-\frac{1}{4v_{1}}<|z|^{2}<-\frac{1}{4v_{2}}\right\},~~Z(I)=\{z_{k}\in\mathcal{Z}:z_{k}\in I\},~~N(I)=|\mathcal{Z}(I)|,\\
Z^{-}(I)=\left\{z_{k}\in\mathcal{Z}:|z|^{2}>-\frac{1}{4v_{2}^{2}}\right\},~~
Z^{+}(I)=\left\{z_{k}\in\mathcal{Z}:|z|^{2}<-\frac{1}{4v_{1}^{2}}\right\},\\
c_{k}(I)=c_{k}\prod_{Rez_{j}\in I_{-}\setminus I}\left(\frac{z_{k}-z_{j}}{z_{k}-z^{*}_{j}}\right)^{2},
\end{gather*}
where $v_{1}\leq v_{2}\in\mathbb{R}^{-}$ are given velocities. Then we define  a distance
\begin{align}
\mu(I)=\min_{z_{k}\in \mathcal{Z}\setminus \mathcal{Z}(I)}\left\{Im(z_{k})\frac{-v_{2}}{|z|^{2}}\left(|z|+\frac{1}{2\sqrt{-v_{1}}}\right)dist(z_{k},I)\right\},
\end{align}
and a space-time cone with given points $y_{1}\leq y_{2}\in\mathbb{R}$
\begin{align}\label{space-time-S}
S(y_{1},y_{2},v_{1},v_{2})=\{(y,t)\in\mathbb{R}^{2},y=y_{0}+vt ~with ~y_{0}\in[y_{1},y_{2}],v\in[v_{1},v_{2}]\}.
\end{align}\\

\centerline{\begin{tikzpicture}[scale=0.6]
\filldraw[pink, line width=0.5](-2,0) arc (0:360:4);
\filldraw[white, line width=0.5](-4,0) arc (-360:0:2);
\draw(-6,0) [black, line width=1] circle(4);
\draw(-6,0) [black, line width=1] circle(2);
\draw[->][thick](-11,0)--(-1,0)[thick]node[right]{$Rez$};
\draw[fill] (-6,2)node[above]{$I$};
\draw[fill] (-5,0)node[below]{$\frac{1}{2\sqrt{-v_1}}$};
\draw[fill] (-1.2,0)node[below]{$\frac{1}{2\sqrt{-v_2}}$};
\draw[fill] (-3,1)node[below]{$z_{5}$} circle [radius=0.08];
\draw[fill] (-3,-1)node[below]{$z^{*}_{5}$} circle [radius=0.08];
\draw[fill] (-6.5,-1)node[left]{$z^{*}_{6}$} circle [radius=0.08];
\draw[fill] (-4,3)node[below]{$z_{7}$} circle [radius=0.08];
\draw[fill] (-4,-3)node[below]{$z^{*}_{7}$} circle [radius=0.08];
\draw[fill] (-6.5,1)node[below]{$z_{6}$} circle [radius=0.08];
\draw[fill] (-8,2)node[below]{$z_{2}$} circle [radius=0.08];
\draw[fill] (-8,-2)node[below]{$z^{*}_{2}$} circle [radius=0.08];
\draw[fill] (-9,1.5)node[below]{$z_{1}$} circle [radius=0.08];
\draw[fill] (-9,-1.5)node[below]{$z^{*}_{1}$} circle [radius=0.08];
\draw[fill] (-10.5,2.8)node[below]{$z_{3}$} circle [radius=0.08];
\draw[fill] (-10.5,-2.8)node[below]{$z^{*}_{3}$} circle [radius=0.08];
\draw[fill] (-7,5)node[below]{$z_{4}$} circle [radius=0.08];
\draw[fill] (-7,-5)node[above]{$z^{*}_{4}$} circle [radius=0.08];
\draw[fill] (-6,-5)node[below]{$(a)$};
\draw[fill] (-4,0) circle [radius=0.08];
\draw[fill] (-2,0) circle [radius=0.08];
\path [fill=pink] (4,0) -- (3,5) to (6.5,5) -- (7,0);
\path [fill=pink] (4,0) -- (4.5,-5) to (8,-5) -- (7,0);
\draw[->][thick](1,0)--(11,0)[thick]node[right]{$y$};
\draw[-][thick](4,0)--(3,5);
\draw[-][thick](4,0)--(4.5,-5);
\draw[-][thick](7,0)--(6.5,5);
\draw[-][thick](7,0)--(8,-5);
\draw[fill] (3.5,0)node[below]{$y_{2}$};
\draw[fill] (5.5,0)node[below]{$S$};
\draw[fill] (7.5,0)node[below]{$y_{1}$};
\draw[fill] (6.5,5)node[right]{$y=v_{2}t+y_{2}$};
\draw[fill] (8,-5)node[right]{$y=v_{1}t+y_{2}$};
\draw[fill] (4.5,-5)node[left]{$y=v_{2}t+y_{1}$};
\draw[fill] (3,5)node[left]{$y=v_{2}t+y_{1}$};
\draw[fill] (6,-5)node[below]{$(b)$};
\end{tikzpicture}}
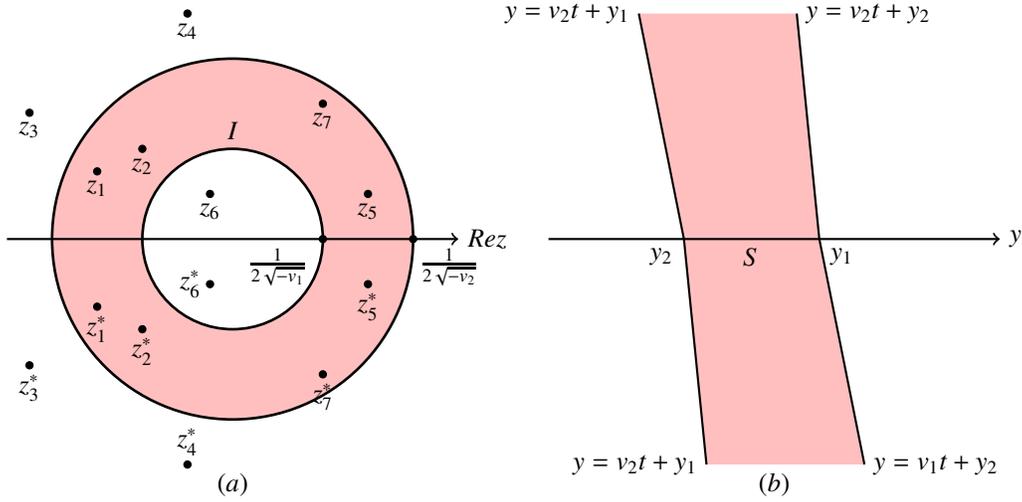
\noindent {\small \textbf{Figure 3.} $(a)$ For example, the original data has nine pairs zero points of discrete spectrum, but insider the cone $S$ only four pairs points with $\mathcal{Z}(I) = {z_1,z_2,z_5,z_7}$; $(b)$ Space-time cone $S(y_{1},y_{2},v_{1},v_{2})$.}

\begin{prop}\label{prop-6.3}
For given scattering data $\sigma_{d}^{\triangle}=\{(z_{k},\hat{c}_{k})\}$,  $t\rightarrow \infty$ and $(y,t)\in S(y_{1},y_{2},v_{1},v_{2})$, we have
\begin{align}\label{I-S}
M^{\triangle_{z_{0},1}^{+}}(z|\sigma_{d}^{\triangle})=\left(\mathbb{I}+O(e^{-2\mu(I) t})\right)M^{\triangle_{z_{0},1}^{+}(I)}
(z|\sigma_{d}(I)),
\end{align}
where
\begin{align}
\sigma_{d}(I)=\{(z_{k},c_{k}(\mathcal{I})s_{22}^{\triangle}(z)^{2}),z_{k}\in \mathcal{Z}(\mathcal{I})\}.
\end{align}
\end{prop}
\begin{proof}
Via employing a similar method to the literature \cite{Yang-SP,AIHP}, the results of this Proposition can be given easily.
\end{proof}

Now, we can derive the asymptotic unique solution $M^{(out)}$ of RHP \ref{RH-6}.

\begin{cor}\label{prop-6.2}
There exist unique solution $M^{(out)}$ of RHP \ref{RH-6}. Particularly,
\begin{align}\label{6.2}
M^{(out)}(z)&=M^{\triangle_{z_{0},1}^{+}}(z)\delta(z)^{-\sigma_{3}}
=M^{\triangle_{z_{0},1}^{+}}(z|\sigma_{d}^{out})\\
&=M^{\triangle_{z_{0},1}^{+}}(z|\sigma_{d}(I))\prod_{Rez_{k}\in I_{-}\setminus I}\left(\frac{z-z_{k}}{z-z^{*}_{k}}\right)^{2}\delta^{-\sigma_{3}}+O(e^{-\mu(I)t}),
\end{align}
where $M^{\triangle_{z_{0},1}^{+}}(z)$ is the solution of RHP \ref{RH-8} with $\triangle=\triangle_{z_{0},1}^{+}$ and $\sigma_{d}^{out}=\{(z_{k},\widetilde{c}_{k}(z_{0}))\}_{k=1}^{N}$ with
\begin{align}
\widetilde{c}_{k}(z_{0})= c_{k}e^{\frac{i}{\pi}\int_{-z_{0}}^{z_{0}}\frac{\log(1+|r(s)|^{2})}{s-z_{k}}ds}.
\end{align}
Substituting Eq.\eqref{6.2} into Eq.\eqref{6.1},  we obtain
\begin{align}\label{6.3}
\|M^{(out)}(z)\|_{L^{\infty}(\mathbb{C}\setminus(\mathcal{Z}\cup\mathcal{Z}^{*}))}\lesssim 1.
\end{align}
In addition,
\begin{align}\label{u-sol-out}
\begin{split}
u_{sol}(y,t;\sigma_{d}^{out})&=\mathop{lim}_{z\rightarrow 0}\frac{\left(M^{(out)}(0)^{-1}M^{(out)}(z)\right)_{12}}{iz},\\
&=u_{sol}(y,t;\sigma_{d}(I))+O(e^{-\mu(I)t}),
\end{split}
\end{align}
where $u_{sol}(y,t;\sigma_{d}^{out})$ is the $N$-soliton solution of Eq.\eqref{1.1} corresponding the scattering data $\sigma_{d}^{out}$.
\end{cor}

\subsection{Local solvable model near phase point $z=\pm z_{0}$}\label{section-local-model}
Based on \eqref{V2-Est-1} and \eqref{V2-Est-2}, it is easily to find that $V^{(2)}-I$ does not have a uniform estimate for large time near the phase point $z=\pm z_{0}$. Therefore, we construct a local solvable model for error function $E(z)$ with a uniformly small jump.

Recall that $\rho=\frac{1}{2}\min_{(z_{a}\neq z_{b})\in \mathcal{Z}\cup \mathcal{Z}^{*}}\{|z_{a}-z_{b}|\}$ and $dist(\mathcal{Z}\cup \mathcal{Z}^{*}, R)>\rho, k=1,2,\ldots,N,$ we find that there are no discrete spectrum in $\mathcal{U}_{\pm z_{0}}$. Consequently, we have $T(z)=\delta(z)$ and RHP \ref{RH-rhp} can be reduced to the following model for the CSP equation \cite{Xu-CSP-JDE}.
\begin{RHP}\label{RH-sp}
Find a matrix value function $M^{sp,+}$, admitting
\begin{itemize}
 \item $M^{sp,+}(y,t;z)$ is continuous in $\mathbb{C}\setminus(\Sigma^{(2)})$.
 \item $M_+^{sp,+}(y,t;z)=M_{-}^{sp,+}(y,t;z)V^{sp}(y,t;z),$ ~~ $z\in\Sigma^{(2)}$, where the jump matrix $V^{sp}(y,t;z)$ satisfies
 \begin{align}
V^{sp}=\left\{\begin{aligned}
&\left(
  \begin{array}{cc}
    1 & r(z_{0})\delta^{-2}(z_{0})(z-z_{0})^{-2i\nu(z_{0})}e^{2it\theta}  \\
    0 & 1 \\
  \end{array}
\right), ~~&z\in\Sigma_{1},\\
&\left(
  \begin{array}{cc}
    1 & 0 \\
    \frac{r^{*}(z_{0})}{1+|r(z_{0})|^{2}}\delta^{2}(z_{0}) (z-z_{0})^{2i\nu(z_{0})}e^{-2it\theta} & 1 \\
  \end{array}
\right), ~~&z\in\Sigma_{2}\cup\Sigma_{9},\\
&\left(
  \begin{array}{cc}
    1 & \frac{r(z_{0})}{1+|r(z_{0})|^{2}}\delta^{-2}(z_{0}) (z-z_{0})^{-2i\nu(z_{0})}e^{2it\theta}  \\
    0 & 1 \\
  \end{array}
\right), ~~&z\in\Sigma_{3}\cup\Sigma_{12},\\
&\left(
  \begin{array}{cc}
    1 & 0 \\
    r^{*}(z_{0})\delta^{2}(z_{0})(z-z_{0})^{2i\nu(z_{0})}e^{-2it\theta} & 1 \\
  \end{array}
\right), ~~&z\in\Sigma_{4},\\
&\left(
  \begin{array}{cc}
    1 & 0 \\
    \frac{r^{*}(-z_{0})}{1+|r(-z_{0})|^{2}}\delta^{2}(-z_{0}) (z+z_{0})^{-2i\nu(-z_{0})}e^{-2it\theta} & 1 \\
  \end{array}
\right), ~~&z\in\Sigma_{5}\cup\Sigma_{10},\\
&\left(
  \begin{array}{cc}
    1 & r(-z_{0})\delta^{-2}(-z_{0})(z+z_{0})^{2i\nu(-z_{0})}e^{2it\theta}  \\
    0 & 1 \\
  \end{array}
\right), ~~&z\in\Sigma_{6},\\
&\left(
  \begin{array}{cc}
    1 & 0 \\
    r^{*}(-z_{0})\delta^{2}(-z_{0})(z+z_{0})^{-2i\nu(-z_{0})}e^{-2it\theta} & 1 \\
  \end{array}
\right), ~~&z\in\Sigma_{7},\\
&\left(
  \begin{array}{cc}
    1 & \frac{r(-z_{0})}{1+|r(-z_{0})|^{2}}\delta^{-2}(-z_{0}) (z+z_{0})^{2i\nu(-z_{0})}e^{2it\theta}  \\
    0 & 1 \\
  \end{array}
\right), ~~&z\in\Sigma_{8}\cup\Sigma_{11};\\
\end{aligned}
\right.
\end{align}
\item $M^{sp}(y,t;z)\rightarrow \mathbb{I},$ \quad $z\rightarrow\infty$.
\end{itemize}
\end{RHP}

Next, we apply the parabolic cylinder(PC) model to solve this problem near the phase point $z=\pm z_{0}$. Unlike the process of solving short pulse equation near the phase point, $M^{sp,+}(y,t;z)$ dose not possess the symmetry that $M^{sp,+}(z;\eta=1)=\sigma_{2}M^{sp,+}(-z;\eta=-1)\sigma_{2}$. Therefore, we have to use PC model to solve the problem near the phase point $z=\pm z_{0}$ separately.\\

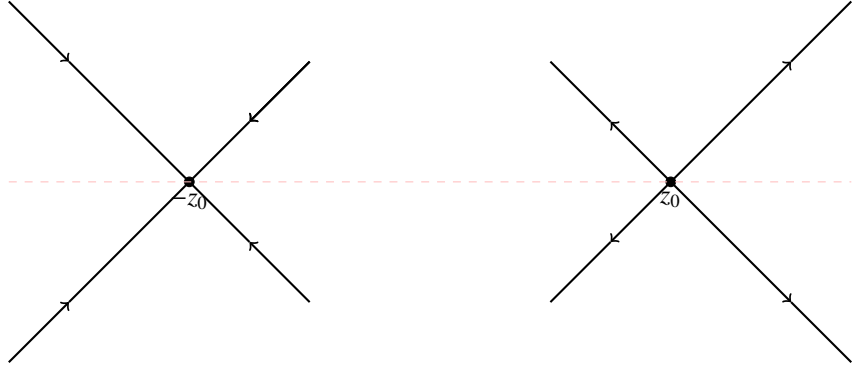
\centerline{\begin{tikzpicture}[scale=0.8]
\draw[->][thick](4,0)--(6,2);
\draw[-][thick](6,2)--(7,3);
\draw[->][thick](4,0)--(6,-2);
\draw[-][thick](6,-2)--(7,-3);
\draw[->][thick](4,0)--(3,1);
\draw[-][thick](3,1)--(2,2);
\draw[->][thick](4,0)--(3,-1);
\draw[-][thick](3,-1)--(2,-2);
\draw[->][thick](-2,2)--(-3,1);
\draw[-][thick](-3,1)--(-6,-2);
\draw[->][thick](-2,2)--(-3,1);
\draw[->][thick](-7,3)--(-6,2);
\draw[->][thick](-2,-2)--(-3,-1);
\draw[-][thick](-3,-1)--(-6,2);
\draw[->][thick](-7,-3)--(-6,-2);
\draw[fill] (4,0)node[below]{$z_{0}$} circle [radius=0.08];
\draw[fill] (-4,0)node[below]{$-z_{0}$} circle [radius=0.08];
\draw[-][pink,dashed](-7,0)--(7,0);
\end{tikzpicture}}
\centerline{\noindent {\small \textbf{Figure 4.}  The jump contour for the local model near the phase point $z=\pm z_{0}$.}}

We first study this model problem near the phase points $z_{0}$. Recall that
\begin{align}\label{7.1}
\delta(z)=\exp\left[i\int_{-z_{0}}^{z_{0}}\frac{\nu(s)}{s-z}ds\right]
=\frac{(z-z_{0})^{i\nu(z_{0})}}{(z+z_{0})^{i\nu(-z_{0})}}e^{\omega(z)},
\end{align}
where $\omega(z)=-\frac{1}{2\pi i}\int_{-z_{0}}^{z_{0}}\log(z-s)d(\log(1+|r(s)|^{2})$.
As $z\rightarrow z_{0}$,
\begin{align}\label{7.2}
\theta(z)=\frac{1}{2z_{0}}+\frac{1}{4z^{3}_{0}}(z-z_{0})^{2}-\frac{1}{4\xi^{4}(z-z_{0})^{3}},
\end{align}
where $\xi$ is a number between $z$ and $z_{0}$.
We evaluate the following scaling transformation
\begin{align}\label{7.3}
(N_{z_{0}}f)(z)=f\left(z_{0}+\sqrt{z_{0}^{3}t^{-1}}z\right),
\end{align}
then, we can derive that
\begin{align}\label{7.4}
(N_{z_{0}}\delta e^{-it\theta(z)})(z)=\delta_{(z_{0})}^{(0)}\delta_{(z_{0})}^{(1)}(z),
\end{align}
where
\begin{gather*}
\delta_{(z_{0})}^{(0)}=(z_{0}^{3}t^{-1})^{\frac{i\nu(z_{0})}{2}}(2z_{0}) ^{-i\nu(-z_{0})}e^{\omega(z_{0})}e^{-\frac{it}{2z_{0}}},\\
\delta_{(z_{0})}^{(1)}(z)=z^{i\nu(z_{0})}
\left(\frac{2z_{0}+\sqrt{z_{0}^{3}t^{-1}}z}{2z_{0}}\right)^{-i\nu(-z_{0})}
e^{\omega\left(z_{0}+\sqrt{z_{0}^{3}t^{-1}}z\right)-\omega(z_{0})}e^{-\frac{iz^{2}}{4}}.
\end{gather*}
From the expression of $\delta_{(z_{0})}^{(1)}(z)$, we can get the conclusion easily that for $\zeta\in\{\zeta=uz_{0}e^{\pm\frac{i\pi}{4}}, -\frac{\rho}{3}<u<\frac{\rho}{3}\}$,
\begin{align}\label{7.5}
\delta_{(z_{0})}^{(1)}(\zeta)\thicksim \zeta^{i\nu(z_{0})}e^{-\frac{i\zeta^{2}}{4}}, ~~ as~~ t\rightarrow +\infty,
\end{align}
from which the influence of the third power can be omitted. Thus, for large $t$, the solution of the Riemann-Hilbert problem for $M^{sp}(y,t;z)$, which is formulated on crosse centered at $z=z_{0}$, can be approximated  based on the PC model see Appendix A.

We introduce the transformation
\begin{align}\label{7.6}
\begin{split}
\lambda&=\lambda(z_0)=\sqrt{\frac{t}{z_{0}^{3}}}(z-z_{0}),\\
r_0=r_{0}^{z_{0}}&=r(z_0)\delta(z_{0})^{-2}e^{2i\left(\nu(z_0)\log(\frac{t}{(z_{0})^{3}})\right)}
e^{\frac{it}{z_{0}^{2}}},
\end{split}
\end{align}
then, the solution $M^{sp,+}(y,t;z)$ formulated on crosse centered at $z=z_{0}$ can be obtained via applying the solution $M^{pc,+}(\lambda)=\sigma M^{(pc),+}(\lambda)\sigma$, shown in Appendix $A$,  where
$\sigma=\left(
         \begin{array}{cc}
           0 & 1 \\
           1 & 0 \\
         \end{array}
       \right)$.
Then, the solution of $M^{sp,+}(y,t;z)$ at $z=z_{0}$ can be expressed as
\begin{align}\label{7.7}
M^{pc,+}(r_{0}^{z_{0}},\lambda)=\mathbb{I}+\frac{M_1^{pc,+}(z_{0})}{i\lambda}+O(\lambda^{-2}),
\end{align}
where
\begin{align*}
M_1^{pc,+}=\begin{pmatrix}0&-\beta_{21}^{z_{0}}(r_{0}^{z_{0}})\\ \beta^{z_{0}}_{12}(r_{0}^{z_{0}})&0\end{pmatrix},
\end{align*}
with
\begin{align*}
\beta^{z_{0}}_{12}=\beta_{12}(r_{0}^{z_{0}})=
\frac{\sqrt{2\pi}e^{i\pi/4}e^{-\pi\nu/2}}{r_{0}^{z_{0}}\Gamma(-i\nu)},\quad \beta_{21}^{z_{0}}=\beta_{21}(r_{0}^{z_{0}})=
\frac{-\sqrt{2\pi}e^{-i\pi/4}e^{-\pi\nu/2}}{(r_{0}^{z_{0}})^*\Gamma(i\nu)}=\frac{\nu}{\beta^{z_{0}}_{12}}.
\end{align*}
By using \eqref{7.6}, we obtain
\begin{align}\label{beta-expre-1}
\beta^{z_{0}}_{12}=\arg\tau(z_{0},+)e^{-4iy-i\nu(z_{0})\log(\frac{t^{2}}{(z_0)^{6}})},
\end{align}
where $|\tau(z_0,+)|^{2}=|\nu(z_{0})^{2}|$ and
\begin{align*}
\arg\tau(z_{0},+)=\frac{\pi}{4}+\arg\Gamma(i\nu(z_{0}))-\arg r(z_{0})- 2\int^{z_{0}}_{-z_0}\log|s-z_{0}|\mathrm{d}\nu(s).
\end{align*}

Furthermore, we consider the model problem near the phase points $-z_{0}$.
For $z\rightarrow -z_{0}$,  we consider the scaling transformation
\begin{align}\label{7.8}
(N_{-z_{0}}f)(z)=f\left(-z_{0}+\sqrt{z_{0}^{3}t^{-1}}z\right),
\end{align}
then, we obtain
\begin{align}\label{7.9}
(N_{-z_{0}}\delta e^{-it\theta(z)})(z)=\delta_{(-z_{0})}^{(0)}\delta_{(-z_{0})}^{(1)}(z),
\end{align}
where
\begin{gather*}
\delta_{(-z_{0})}^{(0)}=(z_{0}^{3}t^{-1})^{-\frac{i\nu(-z_{0})}{2}}(2z_{0}) ^{i\nu(z_{0})}e^{\tilde{\omega}(-z_{0})}e^{\frac{it}{2z_{0}}},\\
\delta_{(-z_{0})}^{(1)}(z)=(-z)^{-i\nu(-z_{0})}
\left(\frac{2z_{0}-\sqrt{z_{0}^{3}t^{-1}}z}{2z_{0}}\right)^{i\nu(z_{0})}
e^{\tilde{\omega}\left(-z_{0}+\sqrt{z_{0}^{3}t^{-1}}z\right)-\tilde{\omega}(-z_{0})}e^{\frac{iz^{2}}{4}}.
\end{gather*}
with
\begin{align*}
\tilde{\omega}(z)=-\frac{1}{2\pi i}\int_{-z_{0}}^{z_{0}}\log(s-z)d(\log(1+|r(s)|^{2}).
\end{align*}
From the expression of $\delta_{(z_{0})}^{(1)}(z)$, we can get the conclusion easily that for $\zeta\in\{\zeta=-uz_{0}e^{\pm\frac{i\pi}{4}}, -\frac{\rho}{3}<u<\frac{\rho}{3}\}$,
\begin{align}\label{7.10}
\delta_{(-z_{0})}^{(1)}(\zeta)\thicksim (-\zeta)^{-i\nu(-z_{0})}e^{\frac{i\zeta^{2}}{4}}, ~~ as~~ t\rightarrow +\infty,
\end{align}
from which the impact of the third power can be omitted. Thus, for large $t$, the solution of the Riemann-Hilbert problem for $M^{sp}(y,t;z)$, which is formulated on crosse centered at $z=-z_{0}$, can be approximated  based on the PC model.

We introduce the transformation
\begin{align}\label{7.11}
\begin{split}
\lambda&=\lambda(-z_0)=\sqrt{\frac{t}{z_{0}^{3}}}(z+z_{0}),\\
r_0=r_{0}^{-z_{0}}&=\frac{r^{*}(-z_{0})}{1+|r(-z_{0})|^{2}}\delta(-z_{0})^{2}
e^{2i\left(\nu(-z_0)\log(\frac{t}{(z_{0})^{3}})\right)}
e^{\frac{it}{z_{0}^{2}}},
\end{split}
\end{align}
then, the solution $M^{sp,+}(y,t;z)$ formulated on crosse centered at $z=-z_{0}$ can be obtained via applying the solution $M^{(pc),+}(\lambda)$ shown in Appendix $A$, i.e.,
\begin{align}\label{7.12}
M^{(pc),+}(r_{0}^{-z_{0}},\lambda)=I+\frac{M_1^{(pc),+}(-z_{0})}{i\lambda}+O(\lambda^{-2}),
\end{align}
where
\begin{align*}
M_1^{(pc),+}(-z_{0})=\begin{pmatrix}0&\beta_{12}^{-z_{0}}(r_{0}^{-z_{0}})\\ -\beta^{-z_{0}}_{21}(r_{0}^{-z_{0}})&0\end{pmatrix},
\end{align*}
with
\begin{align*}
\beta^{-z_{0}}_{12}=\beta_{12}(r_{0}^{-z_{0}})=
\frac{\sqrt{2\pi}e^{i\pi/4}e^{-\pi\nu/2}}{r_{0}^{-z_{0}}\Gamma(-i\nu)},\quad \beta_{21}^{-z_{0}}=\beta_{21}(r_{0}^{-z_{0}})=
\frac{-\sqrt{2\pi}e^{-i\pi/4}e^{-\pi\nu/2}}{(r_{0}^{-z_{0}})^*\Gamma(i\nu)}
=\frac{\nu}{\beta^{-z_{0}}_{12}}.
\end{align*}
By using \eqref{7.11}, we obtain
\begin{align}\label{beta-expre-2}
\beta^{-z_{0}}_{12}=\arg\tau(-z_{0},+)e^{-4iy-i\nu(z_{0})\log\left(\frac{t^{2}}{(z_0)^{6}}\right)},
\end{align}
where $|\tau(-z_0,+)|^{2}=|\nu(-z_{0})^{2}|$ and
\begin{align*}
\arg\tau(-z_{0},+)=\frac{\pi}{4}+\arg\Gamma(i\nu(z_{0}))-
\arg\left(\frac{r^{*}(-z_{0})}{1+|r(-z_{0})|^{2}}\right)- 2\int^{z_{0}}_{-z_0}\log|s+z_{0}|\mathrm{d}\nu(s).
\end{align*}

Noting that the origin is the reference point from which the rays emanate in  model problem, we still use the notation $\lambda$ in the following analysis. Considering that $M^{sp,+}$ admits the asymptotic property
\begin{align}\label{7.13}
M^{sp,+}=\mathbb{I}+\frac{M_1^{pc,+}(z_{0})}{i\lambda}+\frac{M_1^{(pc),+}(-z_{0})}{i\lambda}+O(\lambda^{-2}),
\end{align}
we then substitute the first formula of \eqref{7.6} and \eqref{7.11} into \eqref{7.13}, and obtain
\begin{align}\label{7.14}
M^{sp,+}=\mathbb{I}+\frac{\sqrt{z_{0}^{3}}}{i\sqrt{t}}\frac{M_1^{pc,+}(z_{0})}{z-z_{0}}
+\frac{\sqrt{z_{0}^{3}}}{i\sqrt{t}}\frac{M_1^{(pc),+}(-z_{0})}{z+z_{0}}+O(\lambda^{-2}).
\end{align}
In the local domain $\mathcal{U}_{\pm z_{0}}$, we can obtain the result that
\begin{align}\label{Msp-Est}
|M^{sp,+}-\mathbb{I}|\lesssim O(t^{-\frac{1}{2}}), ~~as~~ t\rightarrow+\infty,
\end{align}
which implies that
\begin{align}\label{7.15}
\|M^{sp,+}(z)\|_{\infty}\lesssim 1.
\end{align}
Since RHP \ref{RH-sp} and \ref{RH-4} possess the same jump conditions in $\mathcal{U}_{\pm z_{0}}$, we apply $M^{sp,+}(z)$ to define a local model in two circles $z\in\mathcal{U}_{\pm z_{0}}$
\begin{align}\label{7.16}
M^{(\pm z_{0})}=M^{(out)}(z)M^{sp,+}(z),
\end{align}
which is a bounded function in $\mathcal{U}_{\pm z_{0}}$ and has the same jump matrix as $M^{(2)}_{RHP}(z)$.

\subsection{The small-norm RHP for $E(z)$}
According to the transformation \eqref{Mrhp}, we have
\begin{align}\label{explict-E(z)}
E(z)=\left\{\begin{aligned}
&M^{(2)}_{RHP}(z)M^{(out)}(z)^{-1}, &&z\in\mathbb{C}\setminus\mathcal{U}_{\pm z_0},\\
&M^{(2)}_{RHP}(z)M^{sp,+}(z)^{-1}M^{(out)}(z)^{-1}, &&z\in\mathcal{U}_{\pm z_0},
\end{aligned} \right.
\end{align}
which is analytic in $\mathbb{C}\setminus\Sigma^{(E)}$ where $\Sigma^{(E)}=\partial\mathcal{U}_{\pm z_0}\bigcup(\Sigma^{(2)}\setminus\mathcal{U}_{\pm z_0})$.\\

\centerline{\begin{tikzpicture}[scale=0.8]
\draw[green, ->][thick](4,0)--(6,2);
\draw[green,-][thick](6,2)--(7,3);
\draw[green,->][thick](4,0)--(6,-2);
\draw[green,-][thick](6,-2)--(7,-3);
\draw[green,->][thick](4,0)--(3,1);
\draw[green,-][thick](3,1)--(2,2);
\draw[green,->][thick](4,0)--(3,-1);
\draw[green,-][thick](3,-1)--(2,-2);
\draw[green,->][thick](2,2)--(1,1);
\draw[green,->][thick](1,1)--(-1,-1);
\draw[green,->][thick](2,-2)--(1,-1);
\draw[green,->][thick](1,-1)--(-1,1);
\draw[green,-][thick](-1,1)--(-2,2);
\draw[green,->][thick](-2,2)--(-3,1);
\draw[green,-][thick](-3,1)--(-6,-2);
\draw[green,->][thick](-2,2)--(-3,1);
\draw[green,->][thick](-7,3)--(-6,2);
\draw[green,-][thick](-1,-1)--(-2,-2);
\draw[green,->][thick](-2,-2)--(-3,-1);
\draw[green,-][thick](-3,-1)--(-6,2);
\draw[green,->][thick](-7,-3)--(-6,-2);
\filldraw[white, line width=0.5](-3.2,0) arc (0:360:0.8);
\filldraw[white, line width=0.5](4.8,0) arc (-360:0:0.8);
\draw [pink, dashed](-8,0)--(8,0);
\draw(4,0) [blue, line width=1] circle(0.8);
\draw(-4,0) [blue, line width=1] circle(0.8);
\draw[fill] (0,0)node[below]{$0$} circle [radius=0.08];
\draw[fill] (4,0)node[below]{$z_{0}$} circle [radius=0.08];
\draw[fill] (-4,0)node[below]{$-z_{0}$} circle [radius=0.08];
\draw[fill] (6,2)node[left]{$\Sigma^{(E)}$};
\draw[fill] (5.5,-0.1)node[above]{$\partial \mathcal {U}_{z_{0}}$};
\draw[fill] (-5.5,-0.1)node[above]{$\partial \mathcal {U}_{-z_{0}}$};
\end{tikzpicture}}
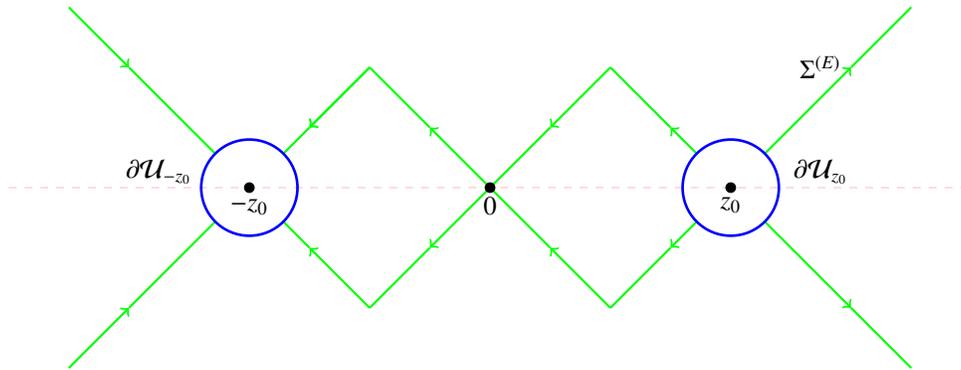
\centerline{\noindent {\small \textbf{Figure 5.}   The jump contour $\Sigma^{(E)}=\partial\mathcal{U}_{\pm z_0}\bigcup(\Sigma^{(2)}\setminus\mathcal{U}_{\pm z_0})$ for the error function $E(z)$.}}

Then it is easy to verify that $E(z)$ admits the Riemann-Hilbert problem.
\begin{RHP}\label{RH-9}
Find a matrix-valued function $E(z)$ such that
\begin{itemize}
\item $E$ is analytical in $\mathbb{C}\setminus\Sigma^{(E)}$;
\item $E^{*}(z^{*})=\sigma_{2}E(z)\sigma_2$;
\item $E(z)=\mathbb{I}+O(z^{-1})$, \quad $z\rightarrow\infty$;
\item $E_+(z)=E_-(z)V^{(E)}(z)$, \quad $z\in\Sigma^{(E)}$, where
\end{itemize}
\begin{align}\label{7.17}
V^{(E)}(z)=\left\{\begin{aligned}
&M^{(out)}(z)V^{(2)}(z)M^{(out)}(z)^{-1}, &&z\in\Sigma^{(2)}\setminus \mathcal{U}_{\pm z_0},\\
&M^{(out)}(z)M^{sp,+}(z)M^{(out)}(z)^{-1}, &&z\in\partial\mathcal{U}_{\pm z_0}.
\end{aligned}\right.
\end{align}
\end{RHP}

By applying Eq.\eqref{V2-Est-1}, Eq.\eqref{V2-Est-2} and Eq.\eqref{6.3}, it is easy to obtain that as $t\rightarrow+\infty$,
\begin{align}\label{VE-I-1}
|V^{(E)}(z)-\mathbb{I}|=\left\{\begin{aligned}
&O\left(e^{-t\frac{\sqrt{2}}{16z_0^{2}}|z\mp z_0|^2}\right) &&z\in\Sigma_{\pm}^{(2)}\setminus\mathcal{U}_{\pm z_0},\\
&O\left(e^{-\frac{t}{4z_0}}\right) &&z\in\Sigma_{0}^{(2)}.
\end{aligned}\right.
\end{align}
While, for $z\in\partial\mathcal{U}_{\pm z_0}$, using Eq.\eqref{6.3} and \eqref{Msp-Est}, we obtain that
\begin{align}\label{VE-I-2}
|V^{(E)}(z)-\mathbb{I}|=|M^{(out)}(z)(M^{sp,+}(z)-\mathbb{I})M^{(out)}(z)^{-1}|=O(t^{-1/2}),~~as ~~t\rightarrow+\infty.
\end{align}
Then, the existence and uniqueness of RHP \ref{RH-9} can be guaranteed by using a small-norm Riemann-Hilbert problem.  Meanwhile, we obtain that
\begin{align}\label{E(z)-solution}
E(z)=\mathbb{I}+\frac{1}{2\pi i}\int_{\Sigma^{(E)}}\frac{(\mathbb{I}+\mu_E(s))(V^{(E)}(s)-\mathbb{I})}{s-z}ds
\end{align}
where $\mu_E\in L^2 (\Sigma^{(E)}) $ and admits
\begin{align}\label{7.18}
(1-C_{\omega_E})\mu_E=\mathbb{I},
\end{align}
where $C_{\omega_E}$ is an integral operator which is defined by
\begin{align*}
C_{\omega_E}f=C_{-}\left(f(V^{(E)}-\mathbb{I})\right),\\
C_{-}f(z)=\lim_{z\rightarrow\Sigma_{-}^{(E)}}\frac{1}{2\pi i}\int_{\Sigma_E}\frac{f(s)}{s-z}ds,
\end{align*}
where $C_{-}$ is the Cauchy projection operator. Then, based on the properties of the Cauchy projection operator $C_{-}$, and the estimate \eqref{VE-I-2}, we obtain that
\begin{align}
\|C_{\omega_E}\|_{L^2(\Sigma^{(E)})}\lesssim\|C_-\|_{L^2(\Sigma^{(E)})}
\|V^{(E)}-\mathbb{I}\|_{L^{\infty}(\Sigma^{(E)})}\lesssim O(t^{-1/2}),
\end{align}
which infers to that $1-C_{\omega_E}$ is invertible which guarantees the existence and uniqueness of $\mu_E$. Then the existence and uniqueness of $E(z)$ are guaranteed. Now, it can be explained that the definition of $M^{(2)}_{RHP}$ is reasonable.

Furthermore, to reconstruct the solutions of $u(y,t)$, the asymptotic behavior of $E(z)$ as $z\rightarrow0$ and large time asymptotic behavior of $E(0)$ is needed. By comparing the estimate \eqref{VE-I-1} with \eqref{VE-I-2}, we find that for $t\rightarrow+\infty$, we only need to consider the calculation on $\partial\mathcal{U}_{\pm z_{0}}$ because it approaches to zero exponentially on other boundary. Then, as $z\rightarrow 0$, we can obtain that
\begin{align}\label{7.19}
E(z)=E(0)+E_{1}z+O(z^{2}),
\end{align}
where
\begin{gather}
E(0)=\mathbb{I}+\frac{1}{2\pi i}\int_{\Sigma^{(E)}}\frac{(\mathbb{I}+\mu_E(s))(V^{(E)}(s)-I)}{s}ds,\label{7.21}\\
E_{1}=-\frac{1}{2\pi i}\int_{\Sigma^{(E)}}\frac{(\mathbb{I}+\mu_E(s))(V^{(E)}(s)-I)}{s^{2}}ds.\label{7.22}
\end{gather}
Then, the large time, i.e., $t\rightarrow+\infty$, asymptotic behavior  of $E(0)$ and $E_{1}$  can be derived as
\begin{align}
E(0)=&\mathbb{I}+\frac{1}{2i\pi}\int_{\partial\mathcal{U}_{\pm z_{0}}}(V^{(E)}(s)-I)ds+o(t^{-1})\notag\\
=&\mathbb{I}+\frac{\sqrt{z_{0}}}{i\sqrt{t}}M^{(out)}(z_0)^{-1}M_1^{pc,+}(z_0)M^{(out)}(z_0)\notag\\
&-\frac{\sqrt{z_{0}}}{i\sqrt{t}}M^{(out)}(-z_0)^{-1}M_1^{(pc),+}(-z_0)M^{(out)}(-z_0)
+\mathcal{O}(t^{-1}),\label{7.24}\\
E_{1}=&\frac{1}{i\sqrt{z_{0}t}}M^{(out)}(z_0)^{-1}M_1^{pc,+}(z_0)M^{(out)}(z_0)\notag\\
&+\frac{1}{i\sqrt{z_{0}t}}M^{(out)}(-z_0)^{-1}M_1^{(pc),+}(-z_0)M^{(out)}(-z_0)
+\mathcal{O}(t^{-1}).\label{7.25}
\end{align}
From \eqref{7.24}, we can derive that
\begin{align}\label{7.26}
E(0)^{-1}=\mathbb{I}+O(t^{-1/2}).
\end{align}

\section{Pure $\bar{\partial}$-RH problem}\label{section-Pure-dbar-RH}
In this section, we study the remaining $\bar{\partial}$-RH problem. The $\bar{\partial}$-RH problem \ref{RH-5} for $M^{(3)}(z)$ is equivalent to the following integral equation
\begin{align}\label{8.1}
M^{(3)}(z)=\mathbb{I}-\frac{1}{\pi}\int_{\mathbb{C}}\frac{M^{(3)}W^{(3)}}{s-z}\mathrm{d}A(s),
\end{align}
where $\mathrm{d}A(s)$ is Lebesgue measure. Further, the equation \eqref{7.1} can be written in operator form
\begin{align}\label{8.2}
(\mathbb{I}-\mathrm{S})M^{(3)}(z)=\mathbb{I},
\end{align}
where $\mathrm{S}$ is Cauchy operator
\begin{align}\label{8.3}
\mathrm{S}[f](z)=-\frac{1}{\pi}\iint_{\mathbb{C}}\frac{f(s)W^{(3)}(s)}{s-z}\mathrm{d}A(s).
\end{align}
We need to prove that the inverse operator $(\mathrm{I}-\mathrm{S})^{-1}$ is invertible, so that the solution $M^{(3)}(z)$ exists.
\begin{lem}
For $t\rightarrow+\infty$, the operator \eqref{8.3} admits that
\begin{align}\label{8.4}
||\mathrm{S}||_{L^{\infty}\rightarrow L^{\infty}}\leq ct^{-1/6}.
\end{align}
where $c$ is a constant.
\end{lem}
\begin{proof}
We mainly prove the case that the matrix function supported in the region $\Omega_1$, the other case can be proved similarly. Denoted that $f\in L^{\infty}(\Omega_1)$, $s=u+iv$ and $z=x+iy$. Then based on \eqref{dbar-R2} and \eqref{5.1}, we can derive that
\begin{align}\label{8.5}
|S[f](z)|&\leq\frac{1}{\pi}\big|f\ \big|_{L^{\infty}(\Omega_{1})}
\iint_{\Omega_{1}}\frac{|M^{(2)}_{RHP}(s)\bar{\partial}R_{1}(s)M^{(2)}_{RHP}(s)^{-1}|}
{|s-z|}df(s)\notag\\
&\leq c\iint_{\Omega_{1}}
\frac{|\bar{\partial}R_{1}(s)||e^{-tv\frac{u^{2}+v^{2}-z_{0}^{2}}{2(u^{2}+v^{2})z_{0}^{2}}}|}{|s-z|}dudv,
\end{align}
where $c$ is a constant.

Based on \eqref{R-estimate-1} and the estimates shown in Appendix $B$, from \eqref{8.5}, we obtain that
\begin{align}\label{8.6}
||\mathrm{S}||_{L^{\infty}\rightarrow L^{\infty}}\leq c(I_{1}+I_{2}+I_{3})\leq ct^{-1/6},
\end{align}
where
\begin{align}\label{8.8}
I_{1}=\iint_{\Omega_{1}}
\frac{|\bar{\partial}\chi_{\mathcal{Z}}(s)|
e^{-tv\frac{u^{2}+v^{2}-z_{0}^{2}}{2(u^{2}+v^{2})z_{0}^{2}}}}{|s-z|}df(s), ~~
I_{2}=\iint_{\Omega_{1}}
\frac{|r'(p)|e^{-tv\frac{u^{2}+v^{2}-z_{0}^{2}}{2(u^{2}+v^{2})z_{0}^{2}}}}{|s-z|}df(s),
\end{align}
and
\begin{align}\label{8.9}
I_{3}=\iint_{\Omega_{1}}
\frac{|s-z_{0}|^{-\frac{1}{2}}e^{-tv\frac{u^{2}+v^{2}-z_{0}^{2}}{2(u^{2}+v^{2})z_{0}^{2}}}
}{|s-z|}df(s).
\end{align}
\end{proof}

Next, our purpose is to reconstruct the large time asymptotic behaviors of $u(x,t)$. According to \eqref{u-sol}, we need the large time asymptotic behaviors of  $M^{(3)}(0)$  and $M_{1}^{(3)}(y,t)$ which are defined as
\begin{align*}
M^{(3)}(z)=M^{(3)}(0)+M_{1}^{(3)}(y,t)z+O(z^{2}),~~z\rightarrow0 ,
\end{align*}
where
\begin{align*}
M^{(3)}(0)=\mathbb{I}-\frac{1}{\pi}\iint_{\mathbb{C}}\frac{M^{(3)}(s)W^{(3)}(s)}{s}
\mathrm{d}A(s),\\
M^{(3)}_{1}(y,t)=\frac{1}{\pi}\int_{\mathbb{C}}\frac{M^{(3)}(s)W^{(3)}(s)}{s^{2}}
\mathrm{d}A(s).
\end{align*}
The $M^{(3)}(0)$ and $M^{(3)}_{1}(y,t)$ satisfy the following lemma.
\begin{lem}\label{prop-M3-Est}
For $t\rightarrow+\infty$,  $M^{(3)}(0)$ and $M^{(3)}_{1}(y,t)$ admit the following inequality
\begin{align}
\|M^{(3)}(0)-\mathbb{I}\|_{L^{\infty}}\lesssim t^{-1},\label{8.10}\\
M^{(3)}_{1}(y,t)\lesssim t^{-1}\label{8.11}.
\end{align}
\end{lem}
The proof of this Lemma is similar to the process that shown in  Appendix $B$.

\section{Soliton resolution for the CSP equation}

Now, we are going to construct the long time asymptotic of the CSP equation \eqref{1.1}.
Recall a series of transformation including \eqref{Trans-1}, \eqref{Trans-2}, \eqref{delate-pure-RHP} and \eqref{Mrhp}, i.e.,
\begin{align*}
M(z)\leftrightarrows M^{(1)}(z)\leftrightarrows M^{(2)}(z)\leftrightarrows M^{(3)}(z) \leftrightarrows E(z),
\end{align*}
we then obtain
\begin{align*}
M(z)=M^{(3)}(z)E(z)M^{(out)}(z)R^{(2)^{-1}}(z)T^{-\sigma_{3}}(z),~~ z\in\mathbb{C}\setminus\mathcal{U}_{\pm z_{0}}.
\end{align*}
In order to recover the solution $u(x,t)$ , we take $z\rightarrow0$ along the imaginary axis which implies $z\in\Omega_{2}$ or $z\in\Omega_{5}$, thus $R^{(2)}(z)=I$. Then, we obtain
\begin{gather*}
M(0)=M^{(3)}(0)E(0)M^{(out)}(0)T^{-\sigma_{3}}(0),\\
M=\left(M^{(3)}(0)+M^{(3)}_{1}z+\cdots\right)\left(E(0)+E_{1}z+\cdots\right)
\left(M^{(out)}(z)\right)\left(T^{-\sigma_{3}}(0)
+\tilde{T}_{1}^{-\sigma_{3}}z+\cdots\right).
\end{gather*}
Based on the above analysis, we can derive that
\begin{align*}
M(0)^{-1}M(z)=&T^{\sigma_{3}}(0)M^{(out)}(0)^{-1}M^{(out)}(z)T^{-\sigma_{3}}(0)z\\
&+ T^{\sigma_{3}}(0)M^{(out)}(0)^{-1}E_{1}M^{(out)}(z)T^{-\sigma_{3}}(0)z\\
&+ T^{\sigma_{3}}(0)M^{(out)}(0)^{-1}M^{(out)}(z)T^{-\sigma_{3}}(0)z+O(t^{-1}).
\end{align*}
Then, according to the reconstruction formula \eqref{u-sol}, \eqref{u-sol-out} and \eqref{7.25}, as $t\rightarrow+\infty$, we obtain that
\begin{align}\label{9.1}
u(x,t)e^{-2d}&=u(y(x,t),t)e^{-2d}\notag\\
&=u_{sol}(y(x,t),t;\sigma_{d}(I))T^{2}(0)(1+T_{1})-it^{-\frac{1}{2}}f^{+}_{12}+O(t^{-1}),
\end{align}
where
\begin{align*}
y(x,t)=x-&c_{+}(x,t,\sigma_{d}(I))-iT_{1}^{-1}-it^{-\frac{1}{2}}f^{+}_{11}+O(t^{-1}),\\
f^{+}_{12}=\frac{1}{i\sqrt{z_{0}}}
&[M^{(out)}(0)^{-1}(M^{(out)}(z_0)^{-1}M_{1}^{pc,+}(z_{0})M^{(out)}(z_0)\\&+
M^{(out)}(-z_0)^{-1}M_{1}^{(pc),+}(-z_{0})M^{(out)}(-z_0))M^{(out)}(0)]_{12},\\
f^{+}_{11}=\frac{1}{i\sqrt{z_{0}}}
&[M^{(out)}(0)^{-1}(M^{(out)}(z_0)^{-1}M_{1}^{pc,+}(z_{0})M^{(out)}(z_0)\\&+
M^{(out)}(-z_0)^{-1}M_{1}^{(pc),+}(-z_{0})M^{(out)}(-z_0))M^{(out)}(0)]_{11}.
\end{align*}

The long time asymptotic behavior \eqref{9.1} gives the solution resolution for the initial value problem of the CSP equation which contains the soliton term confirmed by $N(I)$-soliton on discrete spectrum and the $t^{-\frac{1}{2}}$ order term on continuous spectrum with residual error up to $O(t^{-1})$.

\begin{rem}
The steps in the steepest descent analysis of RHP \ref{RH-2} for $t\rightarrow-\infty$ is similar to the case  $t\rightarrow+\infty$ which has been presented in section $4$-$8$. When we consider $t\rightarrow-\infty$, the main difference can be traced back to the fact that the regions of growth and decay of the exponential factors $e^{2it\theta}$ are reversed, see Fig. 1. Here, we leave the detailed calculations to the interested reader.
\end{rem}

Finally, we can give the results shown in Theorem \ref{Thm-1}

\section*{Acknowledgements}

This work was supported by  the National Natural Science Foundation of China under Grant No. 11975306, the Natural Science Foundation of Jiangsu Province under Grant No. BK20181351, the Six Talent Peaks Project in Jiangsu Province under Grant No. JY-059,  and the Fundamental Research Fund for the Central Universities under the Grant Nos. 2019ZDPY07 and 2019QNA35.

\section{Appendix A: The parabolic cylinder model problem}
Here, we describe the solution of  parabolic cylinder model problem\cite{PC-model,PC-model-2}.
Define the contour $\Sigma^{pc}=\cup_{j=1}^{4}\Sigma_{j}^{pc}$ where
\begin{align}
\Sigma_{j}^{pc}=\left\{\lambda\in\mathbb{C}|\arg\lambda=\frac{2j-1}{4}\pi \right\}.\tag{A.1}
\end{align}
For $r_{0}\in \mathbb{C}$, let $\nu(r)=-\frac{1}{2\pi}\log(1+|r_{0}|^{2})$, we consider the following parabolic cylinder model Riemann-Hilbert problem.
\begin{RHP}\label{PC-model}
Find a matrix-valued function $M^{(pc)}(\lambda)$ such that
\begin{align}
&\bullet \quad M^{(pc)}(\lambda)~ \text{is analytic in}~ \mathbb{C}\setminus\Sigma^{pc}, \tag{A.2}\\
&\bullet \quad M_{+}^{(pc)}(\lambda)=M_{-}^{(pc)}(\lambda)V^{(pc)}(\lambda),\quad
\lambda\in\Sigma^{pc}, \tag{A.3}\\
&\bullet \quad M^{(pc)}(\lambda)=\mathbb{I}+\frac{M_{1}}{\lambda}+O(\lambda^{2}),\quad
\lambda\rightarrow\infty. \tag{A.4}
\end{align}
where
\begin{align}\label{Vpc}
V^{(pc)}(\lambda)=\left\{\begin{aligned}
\lambda^{i\nu\hat{\sigma}_{3}}e^{-\frac{i\lambda^{2}}{4}
\hat{\sigma}_{3}}\left(
                    \begin{array}{cc}
                      1 & 0 \\
                      r_{0} & 1 \\
                    \end{array}
                  \right),\quad \lambda\in\Sigma_{1}^{pc},\\
\lambda^{i\nu\hat{\sigma}_{3}}e^{-\frac{i\lambda^{2}}{4}
\hat{\sigma}_{3}}\left(
                    \begin{array}{cc}
                      1 & \frac{r^{*}_{0}}{1+|r_{0}|^{2}} \\
                      0 & 1 \\
                    \end{array}
                  \right),\quad \lambda\in\Sigma_{2}^{pc},\\
\lambda^{i\nu\hat{\sigma}_{3}}e^{-\frac{i\lambda^{2}}{4}
\hat{\sigma}_{3}}\left(
                    \begin{array}{cc}
                      1 & 0\\
                      \frac{r_{0}}{1+|r_{0}|^{2}} & 1 \\
                    \end{array}
                  \right),\quad \lambda\in\Sigma_{3}^{pc},\\
\lambda^{i\nu\hat{\sigma}_{3}}e^{-\frac{i\lambda^{2}}{4}
\hat{\sigma}_{3}}\left(
                    \begin{array}{cc}
                      1 & r^{*}_{0} \\
                      0 & 1 \\
                    \end{array}
                  \right),\quad \lambda\in\Sigma_{4}^{pc},
\end{aligned}\right.\tag{A.5}
\end{align}
\end{RHP}
\centerline{\begin{tikzpicture}[scale=0.6]
\draw[-][pink,dashed](-4,0)--(4,0);
\draw[-][thick](-4,-4)--(4,4);
\draw[-][thick](-4,4)--(4,-4);
\draw[->][thick](2,2)--(3,3);
\draw[->][thick](-4,4)--(-3,3);
\draw[->][thick](-4,-4)--(-3,-3);
\draw[->][thick](2,-2)--(3,-3);
\draw[fill] (3.2,3)node[below]{$\Sigma_{1}^{pc}$};
\draw[fill] (3.2,-3)node[above]{$\Sigma_{4}^{pc}$};
\draw[fill] (-3.2,3)node[below]{$\Sigma_{2}^{pc}$};
\draw[fill] (-2,-3)node[below]{$\Sigma_{3}^{pc}$};
\draw[fill] (0,0)node[below]{$0$};
\draw[fill] (1,0)node[below]{$\Omega_{6}$};
\draw[fill] (1,0)node[above]{$\Omega_{1}$};
\draw[fill] (0,-1)node[below]{$\Omega_{5}$};
\draw[fill] (0,1)node[above]{$\Omega_{2}$};
\draw[fill] (-1,0)node[below]{$\Omega_{4}$};
\draw[fill] (-1,0)node[above]{$\Omega_{3}$};
\draw[fill] (7,3)node[blue,below]{$\lambda^{i\nu\hat{\sigma}_{3}}e^{-\frac{i\lambda^{2}}{4}\hat{\sigma}_{3}}
\left(
  \begin{array}{cc}
    1 & 0 \\
    r_{0} & 1 \\
  \end{array}
\right)
$};
\draw[blue,fill] (7,-2)node[below]{$\lambda^{i\nu\hat{\sigma}_{3}}e^{-\frac{i\lambda^{2}}{4}\hat{\sigma}_{3}}
\left(
  \begin{array}{cc}
    1 & r^{*}_{0} \\
    0 & 1 \\
  \end{array}
\right)
$};
\draw[blue,fill] (-7,2.5)node[below]{$\lambda^{i\nu\hat{\sigma}_{3}}e^{-\frac{i\lambda^{2}}{4}\hat{\sigma}_{3}}
\left(
  \begin{array}{cc}
    1 & \frac{r^{*}_{0}}{1+|r_{0}|^{2}} \\
    0 & 1 \\
  \end{array}
\right)
$};
\draw[blue,fill] (-7,-1)node[below]{$\lambda^{i\nu\hat{\sigma}_{3}}e^{-\frac{i\lambda^{2}}{4}\hat{\sigma}_{3}}
\left(
  \begin{array}{cc}
    1 & 0 \\
    \frac{r_{0}}{1+|r_{0}|^{2}} & 1 \\
  \end{array}
\right)
$};
\end{tikzpicture}}
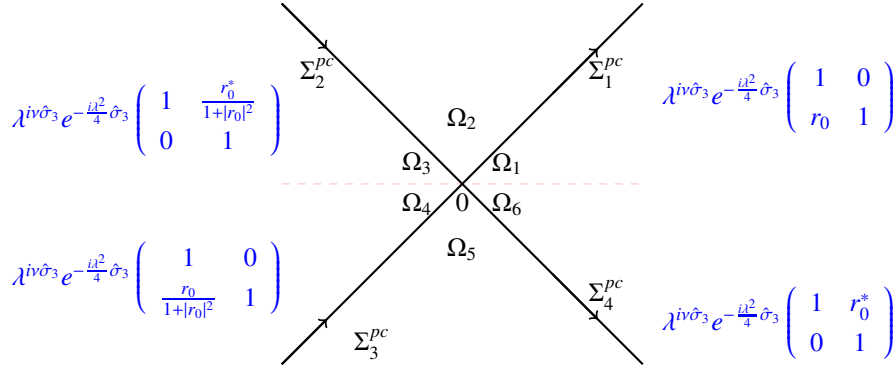
\centerline{\noindent {\small \textbf{Figure 6.} Jump matrix $V^{(pc)}$}.}
We know that the parabolic cylinder equation can be expressed as \cite{PC-equation}
\begin{align*}
\left(\frac{\partial^{2}}{\partial z^{2}}+(\frac{1}{2}-\frac{z^{2}}{2}+a)\right)D_{a}=0.
\end{align*}
As shown in the literature\cite{Deift-1993, PC-solution2}, we obtain the explicit solution $M^{(pc)}(\lambda, r_{0})$:
\begin{align*}
M^{(pc)}(\lambda, r_{0})=\Phi(\lambda, r_{0})\mathcal{P}(\lambda, r_{0})e^{\frac{i}{4}\lambda^{2}\sigma_{3}}\lambda^{-i\nu\sigma_{3}},
\end{align*}
where
\begin{align*}
\mathcal{P}(\lambda, r_{0})=\left\{\begin{aligned}
&\left(
                    \begin{array}{cc}
                      1 & 0 \\
                      -r_{0} & 1 \\
                    \end{array}
                  \right),\quad &\lambda\in\Omega_{1},\\
&\left(
                    \begin{array}{cc}
                      1 & -\frac{r^{*}_{0}}{1+|r_{0}|^{2}} \\
                      0 & 1 \\
                    \end{array}
                  \right),\quad &\lambda\in\Omega_{3},\\
&\left(
                    \begin{array}{cc}
                      1 & 0\\
                      \frac{r_{0}}{1+|r_{0}|^{2}} & 1 \\
                    \end{array}
                  \right),\quad &\lambda\in\Omega_{4},\\
&\left(
                    \begin{array}{cc}
                      1 & r^{*}_{0} \\
                      0 & 1 \\
                    \end{array}
                  \right),\quad &\lambda\in\Omega_{6},\\
&~~~\mathbb{I},\quad &\lambda\in\Omega_{2}\cup\Omega_{5},
\end{aligned}\right.
\end{align*}
and
\begin{align*}
\Phi(\lambda, r_{0})=\left\{\begin{aligned}
\left(
                    \begin{array}{cc}
                      e^{-\frac{3\pi\nu}{4}}D_{i\nu}\left( e^{-\frac{3i\pi}{4}}\lambda\right) & -i\beta_{12}e^{-\frac{\pi}{4}(\nu-i)}D_{-i\nu-1}\left( e^{-\frac{i\pi}{4}}\lambda\right) \\
                      i\beta_{21}e^{-\frac{3\pi(\nu+i)}{4}}D_{i\nu-1}\left( e^{-\frac{3i\pi}{4}}\lambda\right) & e^{\frac{\pi\nu}{4}}D_{-i\nu}\left( e^{-\frac{i\pi}{4}}\lambda\right) \\
                    \end{array}
                  \right),\quad \lambda\in\mathbb{C}^{+},\\
\left(
                    \begin{array}{cc}
                      e^{\frac{\pi\nu}{4}}D_{i\nu}\left( e^{\frac{i\pi}{4}}\lambda\right) & -i\beta_{12}e^{-\frac{3\pi(\nu-i)}{4}}D_{-i\nu-1}\left( e^{\frac{3i\pi}{4}}\lambda\right) \\
                      i\beta_{21}e^{\frac{\pi}{4}(\nu+i)}D_{i\nu-1}\left( e^{\frac{i\pi}{4}}\lambda\right) & e^{-\frac{3\pi\nu}{4}}D_{-i\nu}\left( e^{\frac{3i\pi}{4}}\lambda\right) \\
                    \end{array}
                  \right),\quad \lambda\in\mathbb{C}^{-},
\end{aligned}\right.
\end{align*}
with
\begin{align*}
\beta_{12}=\frac{\sqrt{2\pi}e^{i\pi/4}e^{-\pi\nu/2}}{r_0\Gamma(-i\nu)},\quad \beta_{21}=\frac{-\sqrt{2\pi}e^{-i\pi/4}e^{-\pi\nu/2}}{r_0^*\Gamma(i\nu)}=\frac{\nu}{\beta_{12}}.
\end{align*}
Then, it is not hard to obtain the asymptotic behavior of the solution by using the well known asymptotic behavior of $D_{a}(z)$,
\begin{align}\label{A-1}
M^{(pc)}(r_0,\lambda)=I+\frac{M_1^{(pc)}}{i\lambda}+O(\lambda^{-2}), \tag{A.6}
\end{align}
where
\begin{align*}
M_1^{(pc)}=\begin{pmatrix}0&\beta_{12}\\-\beta_{21}&0\end{pmatrix}.
\end{align*}

\section{Appendix B: Detailed calculations for the pure $\bar{\partial}$-Problem  }
\begin{prop}
For $t>0$ and $z\in\Omega_{1}$, there exists constants $c_{j}(j=1,2,3)$ such that $I_{j}(j=1,2,3)$  which defined in \eqref{8.8} and \eqref{8.9} possess the following estimate
\begin{align}\label{B-1}
I_{j}\leq c_{j}t^{-\frac{1}{6}},~~ j=1,2,3. \tag{B.1}
\end{align}
\end{prop}
\begin{proof}
Let $s=u+iv$ and $z=x+iy$. For $s\in\Omega_{1}$, we know that $\frac{u^{2}+v^{2}-z_{0}^{2}}{(u^{2}+v^{2})z_{0}^{2}}>\frac{v^{2}}{(u^{2}+v^{2})z_{0}^{2}}>0$. Therefore, we assume that there exists an arbitrarily small constant $\varepsilon$ such that $\frac{u^{2}+v^{2}-z_{0}^{2}}{(u^{2}+v^{2})z_{0}^{2}}\geqslant\varepsilon>0$. Then,
using the fact that
\begin{align*}
\Big|\Big|\frac{1}{s-z}\Big|\Big|_{L^{2}}(v+z_{0})=(\int_{v+z_{0}}^{\infty}
\frac{1}{|s-z|^{2}}du)^{\frac{1}{2}}
\leq\frac{\pi}{v-y},
\end{align*}
we can derive that
\begin{align}\label{B-2}
\begin{split}
|I_{1}|&\leq\int_{0}^{+\infty}\int_{v+z_{0}}^{+\infty}
\frac{|\bar{\partial}\chi_{\mathcal{Z}}(s)|e^{-tv\frac{u^{2}+v^{2}-z_{0}^{2}}{2(u^{2}+v^{2})z_{0}^{2}}}}{|s-z|}dudv\\
&\leq\int_{0}^{+\infty}e^{-tv\frac{\varepsilon}{2}}\big|\big|\bar{\partial}\chi_{\mathcal{Z}}(s)\big|\big|_{L^{2}(v+z_{0})}
\Big|\Big|\frac{1}{s-z}\Big|\Big|_{L^{2}(v+z_{0})}dq \\
&\leq \int_{0}^{y}e^{-tv\frac{\varepsilon}{2}}\frac{1}{\sqrt{y-v}}dv
+\int_{y}^{+\infty}e^{-tv\frac{\varepsilon}{2}}\frac{1}{\sqrt{v-y}}dv.
\end{split}\tag{B.2}
\end{align}
Then, using the fact that $e^{-z}\leq z^{-1/6}$, a direct calculation shows that
\begin{align*}
\int_{0}^{y}e^{-tv\frac{\varepsilon}{2}}\frac{1}{\sqrt{y-v}}dv\lesssim t^{-\frac{1}{6}},\\
\int_{y}^{+\infty}e^{-tv\frac{\varepsilon}{2}}\frac{1}{\sqrt{v-y}}dv\lesssim t^{-\frac{1}{2}}.
\end{align*}
Then,  we have $I_{1}\lesssim t^{-\frac{1}{6}}$.
Similarly, considering that $r\in H^{1,1}(\mathbb{R})$, we obtain the estimate
\begin{align}\label{B-3}
|I_{2}|\leq\int_{0}^{+\infty}\int_{v+z_{0}}^{+\infty}
\frac{|r'(u)|e^{-tv\frac{\varepsilon}{2}}}{|s-z|}dudv
\lesssim t^{-\frac{1}{6}}.\tag{B.3}
\end{align}
To obtain the estimate of $I_{3}$, we consider the following $L^{k}(k>2)$ norm
\begin{align}\label{B-4}
\bigg|\bigg|\frac{1}{\sqrt{|s-z_{0}|}}\bigg|\bigg|_{L^{k}}
\leq \left(\int_{v+z_{0}}^{+\infty}
\frac{1}{|u-z_{0}+iv|^{\frac{k}{2}}}du\right)^{\frac{1}{k}}
\leq cv^{\frac{1}{k}-\frac{1}{2}}.\tag{B.4}
\end{align}
Similarly, we can derive that
\begin{align}\label{B-5}
\bigg|\bigg|\frac{1}{|s-z|}\bigg|\bigg|_{L^{k}}\leq c|v-y|^{\frac{1}{k}-1}.\tag{B.5}
\end{align}
By applying \eqref{B-4} and \eqref{B-5}, it is not hard to check that
\begin{align}\label{B-6}
\begin{split}
|I_{3}|&\leq\int_{0}^{+\infty}\int_{v}^{+\infty}
\frac{|z-z_{0}|^{-\frac{1}{2}}e^{-tv\frac{\varepsilon}{2}}}{|s-z|}dudv\\
&\leq\int_{0}^{+\infty}e^{-tv\frac{\varepsilon}{2}}\bigg|\bigg|\frac{1}{\sqrt{|s-z_{0}|}}\bigg|\bigg|_{L^{k}}
\bigg|\bigg|\frac{1}{|s-z|}\bigg|\bigg|_{L^{k}}dv \lesssim t^{-\frac{1}{2}}.
\end{split}\tag{B.6}
\end{align}
Now, we obtain that $I_{1}+I_{2}+I_{3}\lesssim t^{-\frac{1}{6}}$ as $t\rightarrow+\infty$.
\end{proof}

\renewcommand{\baselinestretch}{1.2}


\begin{thebibliography}{00}\addtolength{\itemsep}{-1.5ex}

\bibitem{NLS-optic}
G. P. Agrawal, Nonlinear Fiber Optics. Academic Press, Boston, 1989.
\bibitem{NLS-Maxwell}
A. Hasegawa, Y. Kodama, Solitons in Optical Communications, Oxford University Press, 1995.

\bibitem{Tian-PAMS}
S.F. Tian, T.T. Zhang, Long-time asymptotic behavior for the Gerdjikov-Ivanov type of derivative nonlinear Schr\"{o}dinger equation with time-periodic boundary condition, Proc. Am. Math. Soc. 146 (2018) 1713-1729.
\bibitem{Tian-JDE}
S.F. Tian, Initial-boundary value problems for the general coupled nonlinear Schr\"{o}dinger equation on the interval via the Fokas method, J. Differential Equations, 262 (2017) 506-558.
\bibitem{Tian-PA}
S.F. Tian, The mixed coupled nonlinear Schr\"{o}dinger equation on the half-line via the Fokas method, Proc. R. Soc. Lond. A 472(2195) (2016) 20160588.
\bibitem{Wangds-2019-JDE}
D.S. Wang, B. Guo, X. Wang, Long-time asymptotics of the focusing Kundu-Eckhaus
equation with nonzero boundary conditions, J. Differential Equations, 266(9) (2019) 5209-5253.



\bibitem{NLS-femtosecond}
J.E. Rothenberg, Space-time focusing: breakdown of the slowly varying envelope approximation in the self-focusing of femtosecond pulses, Opt. Lett. 17 (1992) 1340-1342.
\bibitem{SP-Eq}
T. Sch\"{a}fer, C.E. Wayne, Propagation of ultra-short optical pulses in cubic nonlinear media, Phys. D, 196 (2004) 90-105.


\bibitem{mCH-1}
B. Fuchssteiner and A. S. Fokas, Symplectic structures, their B\"{a}cklund transformations
and hereditary symmetries, Phys. D, 4 (1981) 47-66.
\bibitem{mCH-2}
P. J. Olver and P. Rosenau, Tri-Hamiltonian duality between solitons and solitary--wave
solutions having compact support, Phys. Rev. E., 53(1996) 1900-1906.

\bibitem{Constantin-1}
A. Constantin and J. Escher, Wave breaking for nonlinear nonlocal shallow water equations, Acta Math., 181(2) (1998) 229-243.
\bibitem{Constantin-2}
A. Constantin, Existence of permanent and breaking waves for a shallow water equation: A geometric approach, Ann. Inst. Fourier, 50(2) (2000) 321-362.
\bibitem{Constantin-3}
A. Constantin, On the scattering problem for the Camassa-Holm equation, Proc. R. Soc. London, Ser. A, 457(2008) (2001) 953-970.

\bibitem{mCH-3}
Z. Qiao, A new integrable equation with cuspons and $W/M$-shape-peaks solitons, J.
Math. Phys., 47 (2006) 112701.
\bibitem{Fokas-mCH}
A.S. Fokas, On a class of physically important integrable equations, Phys. D, 87(1-4) (1995) 145-150.


\bibitem{SP-no-phy-1}
A. Sakovich, S. Sakovich, Solitary wave solutions of the short pulse equation, J. Phys. A: Math. Gen. 39 (2006) 361-367.
\bibitem{SP-no-phy-2}
Y. Matsuno, Multiloop solutions and multibreather solutions of the short pulse model equation, J. Phys. Soc. Jpn., 76 (2007) 084003.
\bibitem{CSP-equation-Feng}
B.F. Feng, Complex short pulse and couple complex short pulse equations, Phys. D, 297 (2015) 62-75.
\bibitem{Xu-SP-JDE}
J. Xu, Long-time asymptotics for the short pulse equation, J. Differential Equations, 265 (2018) 3439-3532.

\bibitem{CSP-equation-Lax}
A. Sakovich, S. Sakovich, The short pulse equation is integrable, J. Phys. Soc. Jpn., 74 (2005) 239-241.
\bibitem{CSP-equation-rogue}
L. Ling, B.-F. Feng, Z. Zhu, Multi-soliton, multi-breather and higher order rogue wave solutions to the complex short pulse equation, Phys. D, 327 (2016) 13-29.
\bibitem{CSP-conserv-Feng}
B.F. Feng, Complex short pulse and coupled complex short pulse equations, Phys. D, 297 (2015) 085202.
\bibitem{Xu-CSP-JDE}
J. Xu, E.G. Fan, Long-time asymptotic behavior for the complex short pulse equation, J. Differential Equations, 269 (2020) 10322-10349.

\bibitem{Manakov-1974}
S.V. Manakov, Nonlinear Fraunhofer diffraction, Sov. Phys. JETP, 38 (1974) 693-696.
\bibitem{Zakharov-1976}
V.E. Zakharov, S. V. Manakov, Asymptotic behavior of nonlinear wave systems integrated by the inverse scattering method, Sov. Phys. JETP, 44 (1976) 106-112.
\bibitem{Deift-1993}
P. Deift, X. Zhou, A steepest descent method for oscillatory Riemann¨CHilbert problems. Asymptotics for the MKdV equation. Ann. Math. 137(2) (1993) 295-368.
\bibitem{Deift-1994-1}
P. Deift, X. Zhou, Long-time asymptotics for integrable systems. Higher order theory, Comment. Phys. Math.,  165(1) (1994) 175-191
\bibitem{Deift-1994-2}
P. Deift, X. Zhou, Long-Time Behavior of the Non-Focusing Nonlinear Schr\"{o}dinger Equation, a Case Study, Lectures in Mathematical Sciences, Graduate School of Mathematical Sciences, University of Tokyo, 1994.
\bibitem{Deift-2003}
P. Deift, X. Zhou, Long-time asymptotics for solutions of the NLS equation with initial data in a weighted Sobolev space, Commun. Pure Appl. Math. 56(8) (2003) 1029-1077.


\bibitem{McLaughlin-1}
K. T. R. McLaughlin, P. D. Miller, The $\bar{\partial}$ steepest descent method and the asymptotic behavior of polynomials orthogonal on the unit circle with fixed and exponentially
varying non-analytic weights, Int. Math. Res. Not. (2006), Art. ID 48673.
\bibitem{McLaughlin-2}
K. T. R. McLaughlin, P. D. Miller, The $\bar{\partial}$ steepest descent method for orthogonal
polynomials on the real line with varying weights, Int. Math. Res. Not., IMRN (2008), Art. ID 075.

\bibitem{Dieng-2008}
 M. Dieng, K. D. T. McLaughlin, Long-time Asymptotics for the NLS equation via dbar
methods, arXiv: 0805.2807.
\bibitem{Cuccagna-2016}
S. Cuccagna, R. Jenkins, On asymptotic stability of $N$-solitons of the defocusing nonlinear Schr\"{o}dinger equation, Comm. Math. Phys. 343 (2016) 921-969.

\bibitem{AIHP}
M. Borghese, R. Jenkins, K. T. R. McLaughlin, Long-time asymptotic behavior of the
focusing nonlinear Schr\"{o}dinger equation, Ann. I. H. Poincar\'{e} Anal, 35 (2018) 887-920.
\bibitem{Jenkins}
R. Jenkins, J. Liu, P. Perry, C. Sulem, Soliton Resolution for the derivative nonlinear
Schr\"{o}dinger equation, Commun. Math. Phys. 363 (2018) 1003-1049.
\bibitem{Jenkins2}
R. Jenkins, J. Liu, P. Perry, C. Sulem, Global well-posedness for the derivative nonlinear
Schr\"{o}dinger equation, Commun. Part. Diff. Equ. 43(8) (2018) 1151-1195.
\bibitem{Yang-SP}
Y.L. Yang, E.G. Fan, Soliton Resolution for the Short-pluse Equation, arXiv:2005.12208.
\bibitem{Faneg-2}
Q.Y. Cheng, E.G. Fan, Soliton resolution for the focusing Fokas-Lenells equation with weighted Sobolev initial data, arXiv:2010.08714.
\bibitem{Faneg-3}
R.H. Ma, E.G. Fan, Long time asymptotic behavior of the focusing nonlinear Kundu-Eckhaus equation, arXiv:1912.01425.


\bibitem{Li-cgNLS}
Z.Q. Li, S.F. Tian, J.J. Yang, Soliton resolution for a coupled generalized nonlinear
Schr\"{o}dinger equations with weighted Sobolev initial data, arXiv:2012.11928.

\bibitem{PC-model}
A. Its, Asymptotic behavior of the solutions to the nonlinear Schr\"{o}dinger equation, and isomonodromic deformations of systems of linear
differential equations, Dokl. Akad. Nauk SSSR, 261(1) (1981) 14-18.
\bibitem{PC-model-2}
J. Liu, P. Perry, C. Sulem, Long-time behavior of solutions to the derivative nonlinear
Schr\"{o}dinger equation for soliton-free initial data, Ann. I. H. Poincar\'{e}, Anal. Non Lin\'{e}aire, 35 (2018) 217-265.
\bibitem{PC-equation}
F.W.J. Olver, A.B. Olde Daalhuis, D.W. Lozier, B.I. Schneider, R.F. Boisvert, C.W. Clark, B.R. Miller, B.V. Saunders, NIST Digital Library of Mathematical Functions, (2016). http://dlmf.nist.gov/.

\bibitem{PC-solution2}
R. Jenkins, K. McLaughlin, Semiclassical limit of focusing NLS for a family of square barrier initial data, Commun. Pure Appl. Math. 67(2) (2014) 246-320.
\end{thebibliography}
\end{document}